\documentclass[preprint,11pt,authoryear]{elsarticle}
\usepackage{hyperref}
\journal{Stochastic Processes and their Applications}

\usepackage{exscale, extsizes,amsthm,amsmath, textcomp, syntonly, 
			amssymb, amsfonts, graphicx, amsthm, bbm, exscale, 
			textcomp, amsmath, syntonly, amssymb, amsfonts, 
			graphicx, color, amsthm, enumerate, cleveref,
			scrextend, longtable, natbib}
\usepackage{geometry}

\def\sign{\text{sign}}
\newtheorem{Theorem}{Theorem}[section]

\newtheorem{theorem}{Theorem}[section]
\newtheorem{lemma}[Theorem]{Lemma}

\newtheorem{proposition}[Theorem]{Proposition}
\newtheorem{definition}[Theorem]{Definition}

\def\Cov{{\text{Cov}}}

\def\udc{\lceil d \chi \rceil}
\def\:{:\,}

\def\E{{\mathbb E}}

\def\Z{{\mathbb Z}}

\def\real{{\mathbb{R}}}
\def\cal{\mathcal}
\def\definedas{\stackrel{\Delta}{=}}

\def\eqinltwo{\stackrel{L_2}{=}}	

\def\convind{\stackrel{{\cal D}}{\to}}

\def\convinltwo{\stackrel{L_2}{ \to}}

\numberwithin{equation}{section}
\usepackage{etoolbox}

\makeatletter
\patchcmd{\@maketitle}
  {\ifx\@empty\@dedicatory}
  {\ifx\@empty\@date \else {\vskip3ex \centering\footnotesize\@date\par\vskip1ex}\fi
   \ifx\@empty\@dedicatory}
  {}{}
\patchcmd{\@adminfootnotes}
  {\ifx\@empty\@date\else \@footnotetext{\@setdate}\fi}
  {}{}{}
\makeatother

\newcommand{\beq}{\begin{eqnarray}}
\newcommand{\eeq}{\end{eqnarray}}
\newcommand{\beqq}{\begin{eqnarray*}}
\newcommand{\eeqq}{\end{eqnarray*}}

\newcommand*\xbar[1]{%
  \hbox{%
    \vbox{%
      \hrule height 0.5pt 
      \kern0.5ex
      \hbox{%
        \kern-0.2em
        \ensuremath{#1}%
        \kern-0.2em
      }%
    }%
  }%
} 
\begin{document}
	\begin{frontmatter}
		\title{A central limit theorem for the Euler integral of a Gaussian random field}
		\date{\today}
		\author{Gregory Naitzat\footnote{{\it Email address:} {\tt sgregnt@gmail.com}} and Robert J.\ Adler\footnote{{\it Email address:} {\tt robert@ee.technion.ac.il}\\ \hspace*{0.5cm}
		{\it URL:}  {\tt webee.technion.ac.il/people/adler/main.html}}}
		\address{Faculty of Electrical Engineering\\Technion -- Israel Institute of Technology\\
		         Technion City, Haifa 32000, {Israel}}
		\begin{abstract}
			Euler integrals of deterministic functions have recently been shown to have a wide variety of possible applications, including in signal processing, data aggregation and network sensing. Adding random noise to these scenarios, as is natural in the majority of applications,  leads to a need for statistical analysis, the first step of which requires asymptotic distribution results for estimators. The first such result is provided in this paper, as a central limit theorem   for the Euler integral of pure, Gaussian, noise fields.
		\end{abstract}
		\begin{keyword}
			Random field, Euler integral, Gaussian processes, central limit theorem.
			\MSC[2010]  53C65   \sep 60D05\sep   60F05\sep 60G15
		\end{keyword}
	\end{frontmatter}

	\section{Introduction}\label{Sec:intro:rob}

		The Euler characteristic $\chi(A)$ of  a nice set  $A$ is perhaps the oldest, and most fundamental, of its topological invariants. For a compact $A\subset \real^1$, the Euler characteristic is merely the number of its connected components (each one of which will be an interval,  possibly containing only a single point). For $A \subset \real^2$, $\chi(A)$ becomes the number of connected components minus the number of holes, while in three dimensions $\chi(A)$ can be written as  the alternating sum of the numbers of components, handles and hollows. Similar (and, of course, more precise) definitions as alternating sums of Betti numbers, numbers of facets of simplices of differing dimension (when $A$ is triangulisable) or as indices of critical points when a Morse theoretic setting is appropriate, extend the Euler characteristic to a wide variety of sets in arbitrary dimensions.
		
		However, more important for us is that the Euler characteristic is also a valuation, which means that, when all terms are defined. we have the additivity property,
		\beq \label{additivity-rob}
			\chi(A\cup B) = \chi(A)+\chi(B) -\chi(A\cap B).
		\eeq		 
		Given additivity, it is natural to attempt to use $\chi$ to define an integral  on a suitable 
		family of functions, and, indeed, to a large extent this can be done.
		The resulting theory is known as Euler integration. 
		
		\subsection{Euler integration}\label{subsection:intro:EulerIntegral}
		
			Although in many ways Euler integration has its roots in  classical Integral Geometry, a more complete and modern theory began to evolve in the 1970's. More importantly for us, however, is that it has experienced a rapid development in the past decade  from 
			both 	applied and theoretical aspects,  providing for some elegant 
			and novel results. We shall not attempt  to survey these here, since  the recent papers of
			\cite{Ghrist2009} and  \cite{Ghrist2012} provide  excellent and broad expositions. Rather, we shall go directly to two definitions.	
			\begin{definition}
				Let $M\subset \real^n$ be compact, with finite Euler characteristic. Then a continuous function $f\:M\to \real$ is called
				tame if the homotopy 
				types of $f^{-1}((-\infty, u])$ and $f^{-1}([u,\infty))$ change only 
				finitely many times as $u$ varies over $\mathbb{R}$, and the 
				Euler characteristic of each set is  finite.
			\end{definition}
	
			\begin{definition}
				If  $f\:M\to \real$ is tame, then the upper Euler integral of $f$ over $M$ is defined by
				\beq\label{EulerUpper}
					\int_{M} f \lceil d \chi \rceil\ \definedas\  \int_{u=0}^{\infty}[\chi(f > u) - \chi(f \leq -u)]\, du,
				\eeq
				where 
				\beqq
				 \chi(f \leq u) &\definedas& \chi(f^{-1}((-\infty,u])),
				 \eeqq
				 and
				 \beqq
				 \chi(f > u) &\definedas& \chi(M) - \chi(f \leq u).
				 \eeqq
			\end{definition}

			Reading in between the lines that do {\it not} appear in the above definition, one would guess that there is also  a lower Euler integral (there is!) and that there has to be a more direct way to define an integral that follows from the additivity of 
			\eqref{additivity-rob}. In fact, this is also true, and, as a result, the Euler integral 
			shares many common properties with
			the classical theories of integration. However, it is somewhat more delicate, since although \eqref{additivity-rob}
			extends to a finite inclusion-exclusion form, it 	 does not typically extend to the countably infinite case needed for a 
			standard measure based theory of integration. The definition that we have chosen above avoids these issues, and in taking it we follow the lead of  \cite{Bobrowski2011} who, by taking \eqref{EulerUpper} as a definition rather than a property, save often irritating but unimportant (for our needs) technicalities.
	
		\subsection{A motivating application}
		
			An interesting application of the Euler integral is described in 
			\cite{Ghrist2009} and  \cite{Ghrist2012}.  

			Suppose that an unknown number of targets are located in a region $M\subset \real^n$, and each target $\alpha$ is represented by its support $U_{\alpha}\subset M$. Suppose also that the space $M$ is covered with sensors, reporting only the number of targets each one sees. Let $h:X \to \Z$ be the \textit{sensor field}, i.e.\! 
			\beqq
				h(x)\definedas\#\left\{\textrm{targets activating the sensor located at $x$}\right\}.
			\eeqq
			Then, if all the target supports satisfy $\chi(U_{\alpha}) = \beta$ for some $\beta\ne 0$,  the  readings from all the sensors can be combined to obtain the  exact number of targets via the relationship
			\beq
				\label{targets-rob}
				N\ \definedas \ \#\left\{targets\right\} \ =\  \frac{1}{\beta}\int_M h \udc.
			\eeq
			Note that we do not need to assume anything about the targets other than that they all have the same Euler characteristic. For example, we need not assume that they are all convex or even have the same number of connected components.

			While everything in \eqref{targets-rob} is deterministic, 
			\cite{Bobrowski2011} raises the question as to what happens when the deterministic `signal'
			$x= \int_M h \udc$,
			is observed via a noisy measurement $Y=\int_M (h+X)\udc$, where $X$ is a smooth random process on $M$. They show that, although Euler integrals are not always additive, in this case it is true that 
			\beqq
				Y=\int_M (h+X) \udc \ = \  \int_M h \udc + \int_M f \udc \ =\ \beta N + \xi,
			\eeqq
			which leads to the obvious estimator $\hat N$ of $N$ given by   
			\beq
				\label{Nhat-rob}
				\hat N\  =\  \beta^{-1}\left(Y -\E\left[\int_MX \udc\right]\right).
			\eeq
  
			The main result of \cite{Bobrowski2011}   is an elegant calculation of the expectation in \eqref{Nhat-rob}  when $X$ is a smooth Gaussian or Gaussian related random field and $M$ a stratified manifold, based on the Gaussian kinematic formula of  \cite{Adler2007}. Their computation leads to an explicit, closed form (and often somewhat surprising)  expression for the expectation. We shall have more to say about this in Section \ref{sec:meanvalue}.
			
			Motivated by the above, what this paper concentrates on is a central limit theorem (henceforth CLT) needed to go from the estimation provided by $\hat N$ to inference.
 
		\subsection{A CLT for the Euler integral}

			The main result of the paper is formulated  in Theorem \ref{theorem:main}, which states that if 
			$X$ is a real valued, almost surely $C^2$, stationary Gaussian random field on $\mathbb{R}^n$, satisfying certain 
			technical regularity and decay of memory conditions, and if we define
			\beqq
				\Psi_{[0,m]^{n}}[X] \ \definedas \  \int_{[0,m]^n} X\udc,
			\eeqq
			then, as $m \to \infty$,
			\beq 
			\label{equn:clt:rob}
				\frac{\Psi_{[0,m]^{n}}[X] - 
				\E[\Psi_{[0,m]^{n}}[X]]}{m^{{n/2}}}\  \convind \  
				N(0,\sigma_\Psi^2),
			\eeq
			for some limiting variance $\sigma_\Psi^2>0$, where $\convind$ denotes convergence in distribution.
				
		\subsection{Outline of the paper}
		
			Before giving the formal version of \eqref{equn:clt:rob} in Section \ref{Sec:results:rob}, in the following section we shall set up considerable preliminary material, treating the regularity conditions that we require on $X$ as well as background material on Wiener chaos expansions and the Morse theoretical representation of the Euler integral. Then, in Section \ref{Sec:results:rob}, we apply techniques  from \cite{Houdre}  and \cite{Major2014} to develop the  chaos expansion for the upper Euler integral of a Gaussian random field. This, together with some regularity and convergence
			results, are combined  with a general CLT of \cite{Nourdin2012} for chaos expansions,  to make up  the proof our main CLT. Many of the proofs here owe a lot to the papers by \cite{Kratz97}, \cite{Kratz2001}, and especially 	the recent work of \citet{Jose2014}.
			
			In Section \ref{sec:meanvalue} we look at a direct calculation of the mean value of 
			Euler integral (for the isotropic case), showing its dependence on the order-one  Lipshitz-Killing curvature of $M$,  and discuss the surprising results of \cite{Bobrowski2011} alluded to above.	

		\subsection{Acknowledgements}

			1. A reader familiar with the paper \citet{Jose2014} will note a strong resemblance between many of the technical parts of that paper and those in Section \ref{Sec:results:rob} in the present paper. This is not accidental, since that paper develops a CLT for the Euler {\it characteristic} of excursion sets, while we treat a CLT for the Euler {\it integral}. Neither result implies the other, and, as far as we can tell,  neither is derivable from the other without a lot of additional work. Nevertheless, many of the details of the calculations are similar.
			
			We became aware that \citet{Jose2014} was in preparation early in our own work, and are very grateful to Anne Estrade and J\'ose Le\'on not only for sending us a preprint of the final paper, but also for sharing in-progress versions long before the final preprint was ready. Doing so not only saved duplication of effort, but made our lives much easier.
			
			2. 
			Research supported in part by FP7-ICT-318493-STREP and ERC  2012 Advanced Grant 20120216.

	\section{Gaussian random fields and chaos expansions}\label{Sec:prelim:rob}

		Before we set up our results in a formal fashion, we need some preliminaries. In particular, we require a collection of regularity conditions on our random fields which will make the Euler integral well-defined, amenable to analysis, and which are sufficient for our CLT to hold.
		
		In addition, and this will take up most of the section, we need to set up a number of results related to chaos expansions. These will be used in the remainder of the paper to express the Euler integral in this form and then prove our CLT via a general CLT of  \cite{Nourdin2012} for chaos expansions.
		
		For general preliminaries on random fields and their connection to Morse theory we shall use the often complementary books by \citet{Adler2007} and
		 \cite{Az2009}, while for a good treatment of the 	Wiener chaos we rely on \citet{Nualart2006}. Results below that we refer to as ``standard", ``well known", or for which we fail to offer even these descriptions, can be found in one of these references. 

		\subsection{Tame Gaussian fields}
	
			For the remainder of this paper, $X$ will denote a real valued, mean zero, unit variance, Gaussian random field on $\real^n$, $n\geq 1$. 
			We denote its covariance (and correlation) function by $\rho\:\real^n\to\real$.
			For a function $f\:\real^n\to\real$ we denote its gradient by $\nabla f$, writing this and other vectors as row vectors, and its Hessian by $\nabla^2f$. We shall occasionally treat $\nabla^2f$ as a vector rather than a matrix, in which case, because of symmetry, it will have  $n(n+1)/2$ elements. It should be clear from the context whether we are using the matrix or vector interpretations. Generic constants, which may change from line to line, are denoted by $C$.
			
			We write $\Cov (Y)$ for the covariance matrix of a random vector $Y$, and the ubiquitous symbol $|\cdot |$ to denote all of modulus (of a real number), length (of a vector) and determinant (of a matrix). Again, usage should be clear from the context. 	
		
			The regularity conditions we shall require on $X$ are summarised in the following definition.	
	
			\begin{definition}	   
			\label{tame:defn} 			
				Let $X \definedas\ \{X(t),\  t\in \mathbb{R}^n \}$ be as above. Then  we call $X$
				tame if the following conditions all hold.

				\begin{enumerate}[(i)]
					\item At each $t \in \mathbb{R}^n$, the joint distribution of the  vector
							$\langle X(t),\, \nabla X(t),\, \nabla^2 X(t)\rangle $ 
							is non-degenerate.
					
					\item The covariance function, $\rho$, of X is four 
						  	times differentiable, and 		
				  			for some $\alpha>0$, and $t$ small enough, each of its four-order derivatives satisfies
							\beq
							\label{smoothness-rob}\left|	
								\rho^{(4)}(0) - \rho^{(4)}(t)\right|	\ \leq\  
								\frac{C}{(-\ln{|t|})^{1+\alpha}}.
							\eeq
					\item Set 
		          			\beq
		          			\label{psi-rob}
								\psi(t) \ \definedas \  \sup_{0\leq m \leq 4} 
								\left|\frac{\partial^m \rho}{\partial t_{i_1}\dots t_{i_m}}(t)\right|.
		          			\eeq
							Then $\psi \in L^1(\mathbb{R}^n)$, and $\psi(t)\to 0$ as $|t| \to \infty$. 
										$\it{(iv)}$
										
					\item Let $N_v(\nabla X,M)$ be the number of points, $t \in M$, for which $\nabla X(t)= v$. 					
					Then,  for any $v \in \mathbb{R}^n$, 
										\beqq
											\E\left[\left(N_v(\nabla X,M)\right)^3\right] \  < \ \infty.
										\eeqq 
				\end{enumerate}  
			\end{definition}
	
			There are a number of immediate, standard, consequences to tameness for a Gaussian random field. In particular, {\it (ii)} ensures that the trajectories of $X$ are almost  surely (henceforth a.s.) in $C^2(\mathbb{R}^n)$, and, via the exponential integrability of the suprema of Gaussian processes (assured by the Borel-Tsirelson-Ibragimov-Sudakov inequality) that 
			\beqq
				\E\big[\big| \sup_{t \in M} X_t\big|^k\big]  \ <\  \infty,
			\eeqq 	
			for any compact domain $M \subset \mathbb{R}^n$, and any $k\geq 1$.
			
			Condition {\it (i)} ensures that the realisations of $X$ are a.s.\ Morse functions. We shall prove later that {\it (iii)} ensures the decay of correlation necessary for a CLT to hold. Condition {\it(iv)} can be directly verified for  specific covariance functions  using standard integral expressions for the factorial moments of $N_v(\nabla X,M)$, as in 
			\cite{Adler2007}[Theorem 11.5.1]. Alternatively, following \cite{Belyaev1966}, this condition can be substituted by requiring non-degeneracy and smoothness for higher order derivatives of the field. (See also a second moment calculation 
		in an isotropic setting in	\cite{Jose2014}[Proposition 1.1].)
		    					
		\subsection{Correlation structure of stationary, tame, Gaussian fields}\label{decorrelation} 
		
			In what follows, we shall often  need details about the distribution of the  random vector
			\beqq
				\vec{X} \ \definedas \
				\langle
				(\nabla X)_1,\ldots,(\nabla X)_n, X ,(\nabla^2 X)_{1,1},
				(\nabla^2 X)_{1,2}, \ldots, (\nabla^2 X)_{n,n})
				\rangle,
		    \eeqq
			of length $N_n\definedas 1+n+n(n+1)/2$, where		
			\beqq
				(\nabla X)_i = \frac{\partial X}{\partial t_i}, \quad \text{and}\quad  
				(\nabla^2 X)_{i,j} = \frac{\partial^2 X}{\partial t_it_j},  \quad i,j =1,\ldots, n,
			\eeqq
			are the first and second derivatives of $X$. 	
			Since $X$ is tame, all of these derivatives exist, and all are Gaussian. Furthermore, by stationarity, the distribution of 
			$\vec{X}_s$ is independent of $s$, and the elements of the covariance matrix  are given by derivatives of the covariance 
			function $\rho$ at the origin. 
			It is then well established that the covariance matrix $\Lambda$ of $\vec{X}$ factorizes as
			\beq
			\label{equn:lambda-rob}
				\Lambda \ = \
				\begin{pmatrix}
					\Lambda_{(1)} & 0 \\
					0 & \Lambda_{(2)} \\
				\end{pmatrix},
			\eeq
			where $\Lambda_{(1)}$ is the covariance matrix of $\nabla X$  and $\Lambda_{(2)}$ is the covariance matrix of $\langle X,\, \nabla^2 X\rangle$. 	
			Now let $\Lambda^{(1/2)}$ be a square root of $\Lambda$, and define the random field  $Y$ by
			\beq
				\label{equn:Y-rob}
				Y(s)  \ \definedas \ \Lambda^{-(1/2)}\vec{X}(s).	
			\eeq
			 The representation \eqref{equn:lambda-rob} induces a similar factorization  on $\Lambda^{(1/2)}$,  $\Lambda^{-(1/2)}$, and $Y$. 	
			It is important to note that $Y$ is a vector valued random field. Furthermore, since, for each $s$, $Y(s)$ is a vector of independent, standard normal variables,  we shall call $Y$ the {\it decorrelated version} of $\vec{X}$.	
			However, note that despite the independence of the elements of $Y(s)$ for each $s$, the  vectors  $Y(s)$ and $Y(t)$ are
			not independent for $s\neq t$.
	 
			For later needs, note that if we define the covariance matrix  $K$ by
			\beq
				K(t) \ \equiv\ \Cov(Y)(t) \ \definedas \ \E[Y_0^\prime Y_{t}],
			\eeq
	 		then it is easy to check that its entries $\{(K(t))_{ij}\}_{i,j =1}^{N_n}$, 
			are bounded by
			\beq
			\label{equn:boundonK-rob} 
				|(K(t))_{jk} | \notag
				&\leq& \| \Lambda^{-(1/2)} \| \cdot\| \E[\vec{X}_0^\prime\vec{X}_{t}]\|\cdot \|[\Lambda^{-(1/2)}]^\prime\|			
				\\ &\leq&   \notag
				n^2 {\|}\Lambda^{-(1/2)}{\|}^2 
				\cdot\psi(t) \\
				&\leq& C\psi(t).
			\eeq
			Here $\psi$ is given by \eqref{psi-rob} and $C$ is a constant dependent only on the derivatives of $\rho$ at the origin. 
			
			Throughout this work, $\vec{Y}$ denotes the decorrelated version of $\vec X$.
			Before concluding this section and while the definition of $\vec{Y}$ is still fresh in our memory, we introduce
			a new notation, and using this new notation state some of the relations between $\vec{Y}$ and $X$ that we will need later on:
			\begin{itemize}
				\item  For an arbitrary vector $\vec{u}$ of dimension $d$ and a set of indexes $\mathcal{I} = \{i_j\}_{j=1}^{k}$ where $i_j \in \{1,\ldots, n \}$ and $k \leq d$, 	we define  the vector $\vec V_{\mathcal{I}}\left( \vec{u} \right)$ by
				\beq\label{equn:notation_for_vector_pick}
					\vec V_{\mathcal{I}} \left( \vec{u} \right)\ \definedas\ (u_{i_1}, u_{i_2},\ldots,u_{i_k}).
				\eeq
			     In particular, with $ \mathcal{I}=\{N_n - n\} $
			     we have
			     \beqq
			         X = \vec V_{\mathcal{I}} \left(\mathbf{\Lambda}_{(2)}^{(1/2)}\vec{Y}_{(2)}\right),
			     \eeqq
			     and with  $\mathcal{I}=\{m,l\}$
			     \beqq
			         ((\nabla X)_m, (\nabla X)_l) = \vec V_{\mathcal{I}} \left(\mathbf{\Lambda}_{(1)}^{(1/2)}\vec{Y}_{(1)}\right).
			     \eeqq
				\item For an arbitrary vector $\vec{u}$ and a set of indexes $\mathcal{I} = \{i_j\}_{j=1}^{k}$, we implicitly define a symmetric matrix $\mathbf{M}_{\mathcal{I}}\left( \vec{u} \right)$ constructed from the elements of $\vec{u}$ so that, when $\vec u = \left(\mathbf{\Lambda}^{(1/2)}_{(2)} \vec{Y}_{(2)}\right)$ and $\mathcal{I} = \{i_j,\ldots, i_k\}$,  we have
				\beq\label{equn:notation_for_matrix_pick}
					\mathbf{M}_{\mathcal{I}}\left( \mathbf{\Lambda}^{(1/2)}_{(2)} \vec{Y}_{(2)}\right) \definedas\
					\left(
						\begin{matrix}
							({\nabla^2} X)_{i_1i_i}& ({\nabla^2}X)_{i_1i_2} & \cdots & ({\nabla^2} X)_{i_1i_k} \\
							({\nabla^2} X)_{i_2i_i}& ({\nabla^2}X)_{i_2i_2} & \cdots & ({\nabla^2} X)_{i_2i_k} \\
							\vdots &    \vdots  & \cdots &  \vdots \\
							({\nabla^2} X)_{i_ki_i}& ({\nabla^2} X)_{i_ki_2} & \cdots & ({\nabla^2} X)_{i_ki_k} \\
						\end{matrix}
					\right).
				\eeq
				In particular, if $\mathcal{I} = \{1,\ldots, n\}$,	then
				\beq
					\nabla^2 X =   \mathbf{M}_{\mathcal{I}}\left( \mathbf{\Lambda}^{(1/2)}_{(2)} \vec{Y}_{(2)}\right).
				\eeq
			\end{itemize}
		
		\subsection{The spectral distribution of $Y$}\label{Sec:SpecY}
		
			Since $X$, and so its decorrelated  related version $Y$, are stationary, both have spectral representations.  While all properties of the spectral representation of $X$ follow from the classical theory (e.g.\  \citet{Yaglom62}) we need to  work a little to  set up an appropriate representation for the vector valued field $Y$. In particular, we shall do this in the language of isonormal processes, which provides the necessary  structure for later proofs. 
				
			We start by noting that since $X$ is tame, the function $\psi$ is
			integrable, and so by \eqref{equn:boundonK-rob} 	the same is true of the covariance $K$ of $Y$. Consequently, by standard spectral theory, $Y$ has a matrix valued  spectral density  function $f$ for which  
			\beq
			\label{cross-rob}
				(K(\tau))_{jk} \ =\  \E[Y_j(0)Y_k(\tau)] \ =\  
				\int_{\mathbb{R}^n}e^{i\langle \tau,\lambda \rangle}
				f_{jk}(\lambda)d\lambda,\textbf{~~}~~\tau \in \mathbb{R}^{n}.
			\eeq
			It is not too difficult to express the $f_{ij}$ in terms of the spectral density of $X$ (which is, essentially, the only `free parameter' in the entire setup) but, fortunately, their explicit form will not be important in what follows. What is important, however, and follows for the non-degeneracy condition {\it (i)} of tameness, is that,  for all $\lambda \in \mathbb{R}^n$, $(f_{jk}(\lambda)_{j,k=1}^{N_n}$ 
			is a symmetric, positive semi-definite matrix, and so has a 
			symmetric square root $(b_{jk}(\lambda))_{j,k=1}^{N_n}$. Consequently,  we also have that
			\beq
			\label{fjk-rob} 
				f_{jk}(\lambda) \ =\  \sum_{l}b_{jl}(\lambda)b_{lk}(\lambda).
			\eeq
		\subsection{Representing $Y$ via the  isonormal process}\label{Sec:isonormal} 
			Retaining the notation of the previous subsection, we start with a separable Hilbert space $\mathfrak{H}$ of Hermitian functions  
	 		\beq
				\mathfrak{H}  \definedas\ \left\{h(j, \vec \lambda):
				\{1,\ldots, {N_n}\}\times \mathbb{R}^n \to \mathbb{C}~\bigg{|}~\overline{h(j, \vec \lambda)} = h(j, -\vec \lambda),~
				\|h\|^2_{\mathfrak{H}} < \infty\right\},
	 		\eeq		    
 			with the inner product                    
			\beq
				\langle
				h,g
				\rangle_{\mathfrak{H}}  \definedas\  \sum_{j_1=1}^{{N_n}}
				\sum_{j_2=1}^{{N_n}}
				\int_{\mathbb{R}^n}
				h(j_1,\vec \lambda)
				f_{j_1,j_2}(\vec \lambda)
				\overline{g(j_2, \vec \lambda)} d\lambda.
			\eeq
			Next, we let $W^{(j)}$, $j=1,\dots,N_n$ a sequence of independent, real-valued, Gaussian white noises on 
			$\mathbb{R}^n$, and use them to define a random process over $h \in \mathfrak{H}$ by
			\beq
				W(h) \definedas\ \sum_{j=1}^{{N_n}} \sum_{k=1}^{{N_n}}
				    \int_{\mathbb{R}^n} h(j,\vec \lambda) b_{jk}(\vec \lambda) W^{(k)}(d\lambda),
			\eeq
			the integrals here all being standard stochastic integrals. 
 
			By construction $W(h)$ is centered Gaussian and $\E[W(h)W(g)] = \langle f ,g \rangle_{\mathfrak{H}}$. Moreover, since the functions $h$ are Hermitian the resulting random process is real valued. It is known as the {\it isonormal Gaussian process} on $\mathfrak{H}$.

			Finally, we want to relate $Y$ to $W$, as promised.  To this end define a new family of functions, $ \varphi_{t,k}(j,\lambda)$, $	k=1,\ldots,N_n$, $t\in\mathbb{R}^n$,	in $\mathfrak{H}$ via	
			\beq
				\varphi_{t,m} \equiv \varphi_{t,m}(j,\vec \lambda) \definedas 
				\ e^{i\langle\vec  t, \vec \lambda \rangle}\delta_{j,m}\quad j,m=1,\ldots,{N_n},~\vec t,\vec \lambda\in\mathbb{R}^n.
			\eeq
			It is straightforward to check that 			
			\beq
				\notag	 \E[W(\varphi_{\vec t_1,l})W(\varphi_{\vec t_2,m})] =
				\langle \varphi_{\vec t_1,l}, \varphi_{\vec t_2,m} \rangle_{\mathfrak{H}}= \E[Y_m(0)Y_l(t_1-t_2)].
			\eeq
	 		An immediate consequence of this is that the vector valued random field $Y$ has the following particularly  useful $L_2$ representation in terms of the isonormal process and the family  $\varphi_{s,k}$:
	 		\beq
				Y_l(s) \ \eqinltwo\  W(\varphi_{s,l}),
					\qquad 		l =1,\dots N_n, \ \ s \in \mathbb{R}^n.
			\eeq
	 			Note that it also follows from these calculations that 	
	 		\beq
	 	   		\|\varphi_{s,k}\|_\mathfrak{H} \ = \E[Y_k(0)Y_k(0)]\  =\  1,
	 	    \eeq
	 	    and 
	 		\beq	
	 			\langle \varphi_{s,k}, \varphi_{s,m} \rangle_\mathfrak{H} \ = \  \E[W(\varphi_{t,l})W(\varphi_{s,m})] \ = \ 		  \delta_{k,m},
	 		\eeq
	 		where  $\delta_{k,m}$ is the Kronecker delta.
	
		\subsection{Operations on $f\in \mathfrak{H}$}\label{subsection:Operations}
			We now describe the basic operations on $\mathfrak{H}$ that we will need later.
			Let $\{\vec{e}_j\}_{j\geq 1}$ be an orthonormal family of
			functions in $\mathfrak{H}$, and write $\mathbb{M} \definedas \{1,\ldots, {N_n}\}\times
			\mathbb{R}^n$.
			Then, $\vec{e}_j(\lambda): \mathbb{M} \to \mathbb{C}, j\geq
			1$, and
			we define the following operations:
			\begin{itemize}
			  \item Tensor product, $\vec{e}_j \otimes \vec{e}_k:
			        \mathbb{M} \times \mathbb{M} \to \mathbb{C}$:
			        \beq\label{equn:space_M}
			              [\vec{e}_j \otimes \vec{e}_k](\lambda_1, \lambda_2) \definedas
			              \vec{e}_j(\lambda_1) \vec{e}_k(\lambda_2), \quad \lambda_1, \lambda_2 \in \mathbb{M},
			        \eeq
			        where $\vec{e}_j \otimes \vec{e}_k$ belongs to the Hilbert space
			        $\mathfrak{H}^{\otimes 2}$, with the inner product induced, component-wise, by the inner product
			        in $\mathfrak{H}$:
			        \beq
			            \langle \vec{e}_j \otimes \vec{e}_k, \vec{e}_l \otimes
			            \vec{e}_m \rangle_{\mathfrak{H}^{\otimes 2}} \definedas
			            \langle \vec{e}_j,\vec{e}_l\rangle_\mathfrak{H}
			            \langle \vec{e}_k,\vec{e}_m\rangle_\mathfrak{H}\ .
			        \eeq
			  \item In a similar fashion, we define the $m$-fold tensor product of
			        $\vec{e}_j$ with itself:
			        \beq
			              \vec{e}_j^{\otimes m} \definedas\ \underbrace{\vec{e}_j \otimes
			              \cdots \otimes \vec{e}_j}_{\small{m}~\text{times}}\ ,
			        \eeq
			        where $\vec{e}_j^{\otimes m}$ belongs to the Hilbert space $\mathfrak{H}^{\otimes m}$, 
			        with the inner product defined component-wise by the inner product in $\mathfrak{H}$. Likewise, for 
			        $\vec{e}_j^{\otimes q} \in \mathfrak{H}^{\otimes q}$, $\vec{e}_k^{\otimes p} \in \mathfrak{H}^{\otimes p}$, the tensor product of higher order in $\mathfrak{H}^{\otimes q+p}$ is 
			        \beq
			              \vec{e}_j^{\otimes q} \otimes \vec{e}_k^{\otimes p} \definedas\
			              \underbrace{\vec{e}_j \otimes\cdots \otimes
			              \vec{e}_j}_{\small{q ~\text{times}}} \otimes
			              \underbrace{\vec{e}_k \otimes\cdots \otimes
			              \vec{e}_k}_{\small{p~\text{times}}}\ .
			        \eeq
			  \item			
				  Take $0 \leq r  \leq p \leq m$. The $r$-contraction,
				  $(\vec{e}_{j_1} \otimes \cdots \otimes \vec{e}_{j_p})\otimes_r(\vec{e}_{k_{1}}\otimes   \cdots \otimes \vec{e}_{k_m})$,  is in $\mathfrak{H}^{\otimes p + m   -2r}$ and, for $r = 0$, is defined as
                   \beq
                         (\vec{e}_{j_1}\otimes \cdots \otimes
                         \vec{e}_{j_p})\otimes_0 (\vec{e}_{k_{1}} \otimes
                         \cdots \otimes \vec{e}_{k_m}) \ \definedas \ (\vec{e}_{j_1}\otimes \cdots \otimes
                         \vec{e}_{j_p} \otimes \vec{e}_{k_{1}}\otimes
                         \cdots \otimes \vec{e}_{k_m}),
                   \eeq
	               while, for $1 \leq r  \leq p \leq m$,
					\begin{align}
						(\vec{e}_{j_1}\otimes \cdots \otimes
						\vec{e}_{j_p})\otimes_r(\vec{e}_{k_{1}}\otimes
						\cdots \otimes \vec{e}_{k_m}) &\definedas \nonumber \\
						\definedas\left[ \prod_{l=1}^{r}
						\langle \vec{e}_{j_l},\vec{e}_{k_l}\rangle_{\mathfrak{H}} \right]&
						(\vec{e}_{j_{r+1}}\otimes \cdots \otimes
						\vec{e}_{j_p}\otimes\vec{e}_{k_{r+1}}\otimes
						\cdots \otimes \vec{e}_{k_m}).
					\end{align}
					For $r = p = m$ we have 		                            
					\beq
						(\vec{e}_{j_1}\otimes \cdots \otimes
						\vec{e}_{j_p})\otimes_p(\vec{e}_{k_{1}}\otimes
						\cdots \otimes \vec{e}_{k_m}) =
						\left[
						\prod_{l=1}^{r} \langle \vec{e}_{j_l}, \vec{e}_{k_l}\rangle_\mathfrak{\mathfrak{H}}
						\right].
					\eeq
			  \item Symmetrization of $f \in \mathfrak{H}^{\otimes q}$
			        \beq
			          \widetilde{f} \ \equiv \ \textnormal{symm}(f)\ \definedas\
			           \frac{1}{q!}\sum_{\sigma_q}
			           f(\lambda_{\sigma(1)},\cdots,\lambda_{\sigma(q)}),
			        \eeq
			      where $\sigma_q$ is all the permutations over the indexes $\{1,\cdots,q\}$.
			      We write $\mathfrak{H}^{\odot q} \subset \mathfrak{H}^{\otimes q}$
			      for the space of all symmetric
			      $f \in \mathfrak{H}^{\otimes q}$.
			\end{itemize}		   
		
		\subsection{Wiener chaos expansion}\label{Sec:chaosdec}
			Take $W(h)$ to be an isonormal Gaussian process
			on separable Hilbert space $\mathfrak{H}$.
			Write $\mathcal{G} \definedas \sigma(W(h))$, for the $\sigma$-field generated by the random variables $\{W(h), h \in \mathfrak{H}\}$ 
			and $L^2(\mathcal{G}, \mathbb{R})$ for the space of all
			square integrable mappings from $(\Omega, \mathcal{G}, \mathbb{P})$ to $\mathbb{R}$.
				
				       	We make use of
			$\{H_n\}_{n\in \mathbb{N}}$, the probabilistic Hermite
			polynomials, defined by
			\beq
							\exp(tx-\frac{t^2}{2}) =
							\sum_{m=0}^{\infty}{H_m(x)\frac{t^m}{m!}}, \quad x \in \mathbb{R}^n,
			\eeq
			and define
			\beq
			    H_{\vec{a}}(x) \definedas \prod_{j=1}^{\infty}H_{a_j}(x_j),
			    \quad x_j \in \mathbb{R}^n, \quad a_j \in \mathbb{N},~j= 1, \ldots, \infty.
			\eeq
			Here, the sequence $\vec{a} = \{a_1,a_2,\ldots \}$ is such that only a finite
			number of elements differs from zero.
			One can make use of $\{H_{\vec{a}}(x)\}$ to construct
			an orthonormal basis for $L^2(\mathcal{G}, \mathbb{R})$.
			The basis is given by the random variables
			\beq
			    \Phi_{\vec{a}} \definedas \sqrt{\vec{a}!}
			    \prod_{j=1}^{\infty}H_{a_j}(W(\vec{e}_j)),
			    \quad \vec{a}! \definedas \prod_{i=0}^{\infty}a_i!\ ,
			\eeq
			with $W(h), h\in \mathfrak{H}$ the isonormal process as
			defined in Section \ref{Sec:isonormal} and $\{\vec{e}_j\}_{j \geq 1}$
			the orthonormal basis of $\mathfrak{H}$.
				
			Moreover, the space $L^2(\mathcal{G}, \mathbb{R})$ may be represented as the decomposition
			of a countable set of  closed orthogonal subspaces $\{\mathcal{H}_m\}_{m\in \mathbb{N}}$,
			\beq
			    L^2(\mathcal{G}, \mathbb{R}) = \oplus_{m=0}^{\infty}\mathcal{H}_m
			\eeq
			such that, for each $m$, $\{\Phi_{\vec{a}}\}_{|\vec{a}| = m}$ is a complete orthogonal
			system in $\mathcal{H}_m \subset L^2(\mathcal{G}, \mathbb{R})$.	Thus, given $F \in L^2(\mathcal{G}, \mathbb{R})$
			there is a unique decomposition 
			\beq\label{equn:chaos1}
			    F = \sum_{m=0}^{\infty}\sum_{|\vec{a}|=m}\E[\Phi_{\vec{a}} F]\Phi_a,
			     \quad |\vec{a}| \definedas \sum_{i=0}^{\infty} a_i\ ,
			\eeq
			with $\sum_{|\vec{a}|=m}\E[\Phi_{\vec{a}} F]\Phi_{\vec{a}} \in \mathcal{H}_m$,
			orthogonal for different values of $m$.
				
			We are now in  a position to set up the Wiener chaos expansion. To this end define $I_q : \mathfrak{H}^{\odot q} \to L^2(\mathcal{G}, \mathbb{R})$ by
			\beq
			    I_m(\vec{e}_{j_1} \otimes \dots \otimes
			    \vec{e}_{j_q}) = \sqrt{\vec{a}!}
			    \prod_{i=1}^{q} H_{a_i}(W(\vec{e}_{j_i})),
			\eeq
			where $a_i = \#\{k: j_k =i\}$. Using $\{I_q(\cdot)\}$,
			one can rewrite \eqref{equn:chaos1} as
			\beq\label{equn:WCD}
			        F = \sum_{q=0}^{\infty} I_q(f_q), \quad f_q \in \mathfrak{H}^{\odot q}.
			\eeq
			This representation of the random functional $F$ via the set of kernels $\{f_q\}$ through the family of linear operations, $I_q$, is called the {\it Wiener chaos decomposition} of $F$,  and the operator $I_q$ is called the {\it multiple Wiener integral} of order $q$. 		        
			
			If the underlying Hilbert space is a Polish space of the form  $L^2(A,\mathcal{A},\mu)$,  then $I_q$ can be identified with multiple stochastic   integrals. More specifically,  take $\{A_i\}\in \mathcal{A}$ disjoint sets, and define
			\beq
				u_q = \text{symm}\left[\sum a_{i_1,\ldots,i_q}\mathbbm{1}_{A_1}\otimes \cdots \otimes \mathbbm{1}_{A_q}\right].
			\eeq 
		 	The functions $u_q$ are dense in $\text{symm}[\left(L^2(A,\mathcal{A},\mu)\right)^{\otimes q}]$,
			and the integral is constructed first on the functions  $u_q$ by defining
	    	\beq
	    		I(u_q) = \sum a_{i_1,\ldots,i_q}W({A_1}) \cdots W({A_q}),
	    	\eeq 	
	    	and then  extending to a linear continuous operator on all of  $\left(L^2(A,\mathcal{A},\mu)\right)^{\odot q}$. We write					
	    	\beq
	    		I(f_q) =\int_{A} \cdots \int_{A} f_q(t_1,\cdots, t_q ) W(d\mu)\cdots W(d\mu)\quad t_i \in A,
	    	\eeq
	    	with $W(d\mu)$ a Gaussian $\mu$-noise.
				    								
			Below we list some of the basic properties
			of $I_q$ that will be of particular interest to us. For more details, see \cite[Ch. 1]{Nualart2006}. For an in depth treatment related to the approach above, see \cite{Nourdin2012}.
			\begin{itemize}
			    \item $f \in \mathfrak{H}^{\otimes m},\quad I_m(f) = I_m(\widetilde{f})$.
			    \item For any set of $h_i \in \mathfrak{H}$, such that
			          $\|h_i\|_{\mathfrak{H}} = 1, \forall i,$ we have
			          \beq\label{equn:ItoFormula}
			                \prod_{i=1}^{m} H_{a_i}(W(h_i)) =
			                I_{|\vec{a}|}(\textnormal{symm}(h_1^{\otimes a_1} \otimes
			                 \cdots \otimes h_m^{\otimes a_m})).
			          \eeq
			    \item $\E I_p(f)I_q(g) = 0$ when $p \neq q$ and $\E I_p(f)I_q(g) =
			          p!\langle \widetilde{f},\widetilde{g}\rangle_{\mathfrak{H}^{\otimes p}}$ when $p=q$.
			\end{itemize}	
	    	Other properties of the multiple Wiener integral will be recalled when needed. \\

	    	It is in this language of Weiner chaos that we seek to represent the Euler integral, and then make use of the following result.  		
    		\begin{theorem}[Theorem 6.3.1, \cite{Nourdin2012}]\label{theorem:CLTTheoremFirst}
	    		Let $(F_{m})_{m \geq 1}$ be a sequence in
	    		$L^2(\Omega,\mathcal{G},P)$ such that $\E[F_{m}] = 0~~\forall
	    		m$. Suppose that  the chaos expansion of $F$ is given by
	    		$F_m=\sum_{q=1}^{\infty}I_q(f^m_q)$, and suppose in addition that
	    		\begin{enumerate}[(a)]
		    		\item $\quad q!\|f^m_q\|^2_{\mathfrak{H}^{\otimes q}} \to \sigma^2_q$ as $m \to \infty$, for some $\sigma^2_q \geq 0$.
		    		\item $\quad \sigma^2_F \definedas \sum_{q=1}^{\infty}\sigma^2_q < \infty$.
		    		\item $\quad \forall q \geq 2$ and $r =1,\cdots q-1$, $\|f_q^m\otimes_{r}f_q^m\|_{\mathfrak{H}^{\otimes2(q-r)}}^2 \to 0$ as $m \to \infty$.
		    		\item $\quad \lim\limits_{Q \to \infty} \sup_{m \geq 1} \sum_{q=Q+1}^{\infty}q!\|f^m_q\|^2_{\mathfrak{H}^{\otimes q}} \to 0$.
	    		\end{enumerate}
	    		Then, $F_{m} \convind  \mathcal{N}(0,\sigma^2_F)$ as $m \to \infty$.
    		\end{theorem}
	    		
	\subsection{A Rice type formula and a Morse theoretical representation of the Euler integral}\label{Sec:Morse}	
	Although up until now we have approached the Euler integral via integration theory, for the proofs to follow we require a slightly different approach, via stratified Morse theory. This theory links the topology of  sets  to the study of critical points defined on
	them.  		In particular, the Euler characteristic of excursion sets of the form 
		\beqq
									A(M,u) \definedas \{t \in M: f(t)\geq u\},
		\eeqq		
		is easily computed via  properties of the critical points of $f$. 

To see how this works in our setting, we return to  Section \ref{subsection:intro:EulerIntegral}, with $M=T_n \definedas\ [0,m]^n$,  and, noting that $\chi(T_n) = 1$,  obtain
		\beq\label{equn:EulerUpper2}
	        	\int_{T_n} f \lceil d \chi \rceil= \int_{u=0}^{\infty}[1 - \chi(T_n \cap ({f \leq u})) - \chi(T_n \cap ({f \leq -u}))]du.
	    \eeq 	        	      
	        	
		We introduce $\mu(\vec s)$, the Morse index of a critical point $\vec s$  of $f$. (i.e.\ $\nabla f(\vec s)=0$ and the Hessian
		$\nabla^2f(\vec s)$ has $\mu(\vec s)$ negative eigenvalues.) We write $T_n^{\circ}$ for the interior of the cube, then proceed to decompose the boundary of $T_n$ into open faces each of which is an open cube of dimension less then $n$.  (The vertexes  are considered to be zero dimensional closed cubes). To save on notation we denote any such face by $J$,  and write 
		$\{J\}$ to denote the collection of all such faces. When quantities are evaluated with respect to a particular face, this will be denoted by  an appropriate subscript. With the above notation, a careful application of  Theorem 9.3.5   of \cite{Adler2007}
		(see also \cite{Bobrowski2011})  yields 	        	
		\begin{multline}\label{equn:morseinterpretation}
			\int_{T_n} f \lceil\chi \rceil =    \sum_{\{\vec s\in T_n^{\circ}:\nabla f(\vec s) = \vec 0\}}(-1)^{\mu(\vec s)}f(\vec s)\ +\\
			+ \sum_{ J  \in \{J\}\backslash T_n^{\circ}}
			\left(
			\sum_{\{\vec s \in J: \nabla f_{|J}(\vec s) = \vec 0\}}(-1)^{\dim J - \mu_{f_{|J}}(\vec s)}f(\vec s)
			\mathbbm{1}_{\{ \langle \nabla f(\vec{s}), \vec \eta_J\rangle \geq 0\}}
			\right).
		\end{multline}
	        		        
		In the above, $\{\vec \eta_J\}$ are constant vectors attached to every face of the cube $T_n$ (the details can be found in \cite{Adler2007}). 
		We identify 
		\beq\label{equn:interiorSum}
			\sum_{\{\vec s\in T_n^{\circ}:\nabla f(\vec s) = \vec 0\}}(-1)^{\mu(\vec s)}f(\vec s)
		\eeq
		as the contribution of the internal critical points to the Euler integral, and
		\beq\label{equn:boundarySum}
			\sum_{ J  \in \{J\}\backslash T_n^{\circ}}
			\left(
			\sum_{\{\vec s \in J: \nabla f_{|J}(\vec s) = \vec 0\}}(-1)^{\dim J - \mu_{f_{|J}}(\vec s)}f(\vec s)
			\mathbbm{1}_{\{\langle \nabla f(\vec{s}), \vec \eta_J\rangle \geq 0\}}
			\right)
		\eeq
		as the contribution of the critical points on the boundary.
		          
		          From a critical point representation of this kind, one can develop an integral representation of Rice type, and it is this that will be at the core of all the proofs to follow.
		                                       
	    \begin{lemma}\label{lemma:lemma0}
	        Let $f$ be a Morse function. Then 
	        	        \begin{multline}\label{equn:GApprox}
	            \int_{T_n} f \lceil d\chi \rceil =  
				\lim\limits_{\sigma \to 0}
				\int_{T_n}
				{\phi_{\sigma^2 \mathbf{I}_{n\times n}}{(\nabla{{f}_{|T_n^{\circ}}(\vec s)})\det(\nabla^2{f}_{|T_n^{\circ}}(\vec s))f_{|T_n^{\circ}}(\vec s)
				}
				d\vec s}~+ \\ +
				\sum_{ J  \in \{J\}\backslash T_n^{\circ}}(-1)^{\dim J}             \lim\limits_{\sigma \to 0}
				\int_{J}
				\phi_{\sigma^2 \mathbf{I}_{\dim J \times \dim J}}{(\nabla{{f}_{|J}(\vec s)})\det(\nabla^2{f}_{|J}(\vec s))f_{|J}(\vec s)
				\mathbbm{1}_{\{\langle \nabla f(\vec{s}), \vec{\eta}_J\rangle \geq 0\}}
				d\vec s},
	       \end{multline}
	       where $\phi_{\sigma^2 \mathbf{I}_{\dim J \times \dim J}}(\vec s)$ is a
	       $(\dim J)$-dimensional centered Gaussian kernel with the covariance matrix
	       $\sigma^2 \mathbf{I}_{\dim J \times \dim J}$.       
	    \end{lemma}

	\begin{proof}
		Standard techniques for the construction of Rice type integral formulae, along with the fact that  
		the  $\sigma \to 0$ limit of  $\phi_{\sigma^2 \mathbf{I}}$ is a Dirac delta,  establish that, for any $J \in \{J\}$,
				\begin{align}\label{equn:delta_composition1}
				\notag
			\lim_{\sigma \to 0} 
			\int_{T_n} 
			\phi_{\sigma^2 \mathbf{I}}
			(\nabla f(\vec s))\left|\det(\nabla^2 f(\vec s))\right| f(\vec s)d\vec s & = \int_{T_n} (\delta \circ  \nabla f)(\vec s) f(\vec s)d\vec s \\ &= \sum_{\{\vec s \in T_n^{\circ}:\nabla f(\vec s) = \vec 0\}}f(\vec s),
		\end{align}
		and
		\begin{align}\label{equn:delta_composition2}  \notag
			&\lim_{\sigma \to 0} \int_{J} \phi_{\sigma^2 \mathbf{I}_{\dim J \times \dim J}}(\nabla{{f}_{|J }(\vec s)})\left|\det(\nabla^2{f}_{| J}(\vec s))\right|f_{|J}(\vec s) \mathbbm{1}_{\{\langle \nabla f(\vec{s}), \vec{\eta}_J\rangle \geq 0\}}d\vec s  \\
			&\qquad = \int_{J} (\delta \circ  \nabla f_{| J})(\vec s) f_{| J}(\vec s)  \mathbbm{1}_{\{\langle \nabla f(\vec{s}), \vec{\eta}_J\rangle \geq 0\}}d\vec s  \notag
			\\ &\qquad  = \sum_{\{\vec s \in J:\nabla f_{|J}(\vec s) = \vec 0\}} f(\vec s) \mathbbm{1}_{\{\langle \nabla f(\vec{s}), \vec{\eta}_J\rangle \geq 0\}}.
			\end{align}
		Using the fact that the  determinant of a matrix equals the product of its eigenvalues, 
		it follows  that $\sign\{\det\left(\nabla^2 f_{|J}\right)(\vec s)\} = (-1)^{\mu_{|J(\vec s)}}$. Thus,  we can drop the absolute value in \eqref{equn:delta_composition1} and \eqref{equn:delta_composition2},  and, applying \eqref{equn:morseinterpretation},  complete  the proof.
	\end{proof}
	
	We now have all that we need to formulate, and to prove, the main result of this paper.
                                      			  
\section{A CLT for the Euler integral}\label{Sec:results:rob}

	\begin{theorem}\label{theorem:main}
		Let $X \definedas \{X(\vec{s}) | \vec{s}\in \mathbb{R}^n \}$ be a tame Gaussian field, as in Definition
		\ref{tame:defn}.
		Then, the (upper)
		Euler integral
		\beqq
			\Psi_{[0,m]^{n}}[X]\ \definedas\
			\int_{[0,m]^{n}}X(\vec s) \lceil d\chi \rceil
		\eeqq
		satisfies the  central limit theorem
		\beqq
			\frac{\Psi_{[0,m]^{n}}[X] -
			\E[\Psi_{[0,m]^{n}}[X]]}{m^{{n}/{2}}} \ \convind \
			\mathcal{N}(0,\sigma^2_\Psi),\text{ as } m \rightarrow \infty.
		\eeqq
		where  $\sigma^2_\Psi>0 $ is defined by \eqref{sigmapsi:eqn} below.
	\end{theorem}        
   Note that an expression for the mean value of Euler integral, $\E[\Psi_{[0,m]^{n}}[X]]$, was derived in \cite{Bobrowski2011},  and is also discussed in Section \ref{mean:sec}  below from the point of view of chaos expansions.

   Before starting the proof of Theorem \ref{theorem:main}, note that while expressions like    \eqref{equn:morseinterpretation}  and  \eqref{equn:GApprox}   relate to the full Euler integral, only the first sum in \eqref{equn:morseinterpretation}  and  the first integral
    in \eqref{equn:GApprox}, which relate to contributions from the interior of $T_n$,   are relevant for the CLT.  The reason for this lies in the normalisation of $m^{-{n/2}}$, which applies equally to all terms. It  follows from the calculations of this subsection that all non-interior terms, when normalised by $m^{-{n/2}}$,  converge in probability to zero, and so not affect the limiting distribution. We leave the (simple) details of this to the reader, and so in dealing with the CLT henceforth concentrate only on the interior terms.

   We start with a sequence of lemmas, which will ultimately be combined to provide a full proof of Theorem \ref{theorem:main} in the following subsection.
   
   \subsection{Four supporting lemmas}
   
    \begin{lemma}\label{lemma:lemma1}
		Let $X$ be a tame Gaussian random field. Let 
		\beq
		 	F_{(0,m)^n}[X] = \sum_{\{\vec s\in T_n^{\circ}:\nabla X(\vec s) = \vec 0\}}(-1)^{\mu(\vec s)}X(\vec s),
		\eeq
		and
		\beq
		 F_{(0,m)^n}^{\sigma}[X] =
		     \int_{(0,m)^n}
		         {\phi_{\sigma^2 \mathbf{I}}(\nabla{X}(\vec s))\det(\nabla^2{X}(\vec s))X(\vec s)d\vec s}.
		\eeq
		Then,
		\beq
		       F_{(0,m)^n}^{\sigma}[X]
		       \convinltwo F_{(0,m)^n}[X]~~\text{as}~~ \sigma \to 0.
		\eeq
	\end{lemma}

	\begin{proof}
		We deduce the $L^2$ convergence from the  following  two facts:
		\begin{enumerate}[(a)]
			\item $F_{(0,m)^n}^{\sigma}[X]
			\overset {a.s.}\longrightarrow F_{(0,m)^n}[X]
			\text{ as } \sigma \to 0$,
			\item There exists $\varepsilon>0$ so that
			$\sup_{\sigma}\E \big{[}
			F_{(0,m)^n}^{\sigma}[X]
			\big{]}^{2 + \varepsilon} <\infty$.
		\end{enumerate}
		
		To prove (a), note that the trajectory of $X$ is almost surely Morse and the result then follows from Lemma \ref{lemma:lemma0}. 
		To show (b), we write an upper bound for $\E \big{[}
		F_{(0,m)^n}^{\sigma}[X]
		\big{]}^{2 + \varepsilon}$ independent of $\sigma$. To this end, using Federer's coarea formula (cf.\ \cite{Az2009},
		Proposition 6.1, for a version couched in our terminology) we have
		\beq\label{equn:aria}
			\int_{(0,m)^n}{\phi_{\sigma^2 \mathbf{I}}(\nabla{X}(\vec s))|
			\det(\nabla^2{X}(\vec s))|d\vec s} =
			\int_{\mathbb{R}^n}{\phi_{\sigma^2 \mathbf{I}}(\vec u)
			N_{\vec u}(\nabla{X},(0,m)^n)d \vec u}.
		\eeq 
		Write $X_{\text{sup}}\definedas \sup_{\vec s\in(0,m)^n}|X(\vec s)|$. Then,
		applying  \eqref{equn:aria}, we obtain
		\beqq
			\int_{(0,m)^n}{\phi_{\sigma^2 \mathbf{I}}(\nabla{X}(\vec s))
			\det(\nabla^2{X}(\vec s))X(\vec s)d\vec s}
			\leq
			\left|(X_{\text{sup}})\right|
			\left|
			\int_{\mathbb{R}^n}{\phi_{\sigma^2 \mathbf{I}}(\vec u)N_{\vec u}
			(\nabla{X},(0,m)^n)d\vec u}
			\right|.
		\eeqq
		Using the fact that $ab \leq \frac{a^p}{p} + \frac{b^q}{q}$ for $a,b>0$ and $\frac{1}{p}+\frac{1}{q}=1$, we have  
		\begin{align}                 
			\left|(X_{\text{sup}})                                                          \int_{\mathbb{R}^n}{\phi_{\sigma^2 \mathbf{I}}(\vec u)N_{\vec u}
			(\nabla{X},(0,m)^n)d\vec u}
			\right|^{(2+\varepsilon) }&  \notag \\
			\leq
			\frac{1}{p} \left|(X_{\text{sup}})^{(p(2+\varepsilon))}\right| + &
			\frac{1}{q}\left|
			\int_{\mathbb{R}^n}{\phi_{\sigma^2 \mathbf{I}}(\vec u)N_{\vec u}
			(\nabla{X},(0,m)^n)d\vec u}
			\right|^{q(2+\varepsilon)}.
		\end{align}
		Taking  expectations yields    
		\begin{align}\label{equn:after_holder}
			\E\bigg{[}\left|
			\int_{(0,m)^n}{\phi_{\sigma^2 \mathbf{I}}(\nabla{X}(\vec s))
			\det(\nabla^2{X}(\vec s))X(\vec s)d\vec s}\right|
			\bigg{]}^{(2+\varepsilon)}&
			\notag \\ \leq \frac{1}{p}
			\E \big{|}(X_{\text{sup}})\big{|}^{p(2+\varepsilon)} +
			\frac{1}{q}\E\bigg{[}&
			\int_{\mathbb{R}^n}{\phi_{\sigma^2 \mathbf{I}}(\vec u) N_{\vec u}(\nabla{X},(0,m)^n)d\vec u}                                                                \bigg{]}^{q(2+\varepsilon)}.
		\end{align}
		$\E \big{|}(X_{\text{sup}})\big{|}^{p(2+\varepsilon)}$ is finite due to our assumptions of tameness on $X$, so we focus on the second term in \eqref{equn:after_holder}, viz. 
		\beq
			\E\bigg{[}
			\int_{\mathbb{R}^n}{
			\phi_{\sigma^2 \mathbf{I}}(\vec u) N_{\vec u}(\nabla{X},(0,m)^n)d\vec u}                                                                \bigg{]}^{q(2+\varepsilon)}.
		\eeq
		Jensen's inequality, when applied to
		the inner integral (and not to the expectation), implies that the above can be bounded by
		\beqq                                
			\E\bigg{[}
			\int_{\mathbb{R}^n}{\phi_{\sigma^2 \mathbf{I}}(\vec u)\big{[}
			N_{\vec u}(\nabla{X},(0,m)^n)\big{]}^{(2+\varepsilon)}d\vec u}
			\bigg{]}.
		\eeqq                             
		Using Tonelli's theorem, this equals
		\beqq 
			\int_{\mathbb{R}^n}{\phi_{\sigma^2 \mathbf{I}}(\vec u)
			\E\big{[}
			N_{\vec u}(\nabla{X},(0,m)^n)\big{]}^{q(2+\varepsilon)}
			d\vec u},
		\eeqq 
		and, finally (for $q = 1 + \varepsilon)$ under the assumptions of a tameness we have
		\beqq 
			\E
			N_{\vec u}(\nabla{X},(0,m)^n)
			^{(2+\varepsilon)} \leq M,
		\eeqq 
		and we are done.\\

		For the other faces $J \in \{J\} \backslash T^{\circ}_n$, of dimension $d \equiv$ dim$(J)<n$, we have
		\begin{align}
			\left|
			\E\bigg{[}
			\int_{J}{\phi_{\sigma^2 \mathbf{I}}(\nabla{X}(\vec s))
			\det(\nabla^2{X}(\vec s))X(\vec s) \mathbbm{1}_{\{\langle \nabla X(\vec{s}), \eta_J\rangle \geq 0\}}d\vec s}
			\bigg{]}^{(2+\varepsilon)}
			\right| 
			\qquad \qquad \notag  \\ 
			\qquad \qquad \qquad 
			\leq                           
			\E\bigg{[}
			(X_{\text{sup}})
			\int_{\mathbb{R}^d}{\phi_{\sigma^2 \mathbf{I}}(\vec u)N_{\vec u}
			(\nabla{X},J)d\vec u}
			\bigg{]}^{(2+\varepsilon)},
		\end{align}
		and can then repeat the same argument as above.
	\end{proof}

	For the next lemma, which deals with  the Wiener chaos decomposition of $F_{(0,m)^n}[X]$, we introduce the notations 
	$$\pi_n(q) \equiv \big{\{} \vec{a}~ \big{|}~a_1 +...+ a_{N_n} = q\big{\}},$$
	 and 
	 $$\widetilde{H}_{\vec{a}}(Y_s) \equiv \prod_{i=1}^{N_n} H_{a_i}(Y_i(s)).$$

	\begin{lemma}\label{lemma:lemma2}
		$F_{(0,m)^n}^{\sigma}[X]$ admits
		the  Wiener chaos expansion
		\beq\label{equn:PhiExp1}
			F_{(0,m)^n}^{\sigma}[X] \ \eqinltwo\
			\sum_{q=1}^{\infty}
			\sum_{\vec{a} \in \pi_n(q)}
			d^{\sigma}_{\vec{a}}\int_{\small{(0,m)^n}}
			\widetilde{H}_{\vec{a}}(\vec{Y}_{\vec s})d\vec s,
		\eeq
		where
				\begin{align}
			d^{\sigma}_{\vec{a}} &=
			\frac{1}{\vec{a}!}
			\int_{\mathbb{R}^n}\phi_{\sigma^2 \mathbf{I}_{n\times n}}\left(\mathbf{\Lambda}_{(1)}^{(1/2)}\vec{v}\right)
			\prod_{i=1}^{n} H_{a_i}(v_i)\phi(v_i)d \vec v\\
			&\times
			\int_{\mathbb{R}^{({N_n}-n)}}
			{ \det\left( \mathbf{M}_{\mathcal{I}}\left( \mathbf{\Lambda}^{(1/2)}_{(2)} \vec{u}\right)\right)
			\vec V_{\left\{N_n - n\right\}}\left(\mathbf{\Lambda}_{(2)}^{(1/2)}\vec{u}\right)
			\prod_{i=1}^{{N_n}-n} H_{a_{n+i}}(u_i)\phi(u_i)d \vec u}.
		\end{align}
    \end{lemma}

	\begin{proof}
		Consider
		\beqq
			F^{\sigma}_{(0,m)^n}[X] =\int_{(0,m)^n}
			{\phi_{\sigma^2 \mathbf{I}}
			(\mathbf{\Lambda}_{(1)}^{(1/2)}\vec{Y}_{(1)})
			\det\left( \mathbf{M}_{\mathcal{I}}\left( \mathbf{\Lambda}^{(1/2)}_{(2)} \vec{Y}\right)\right)
			\vec V_{\left\{N_n - n\right\}}\left(\mathbf{\Lambda}_{(2)}^{(1/2)}\vec{Y}\right)
			d \vec s}.
		\eeqq 
		Take $f_1: \mathbb{R}^n \to \mathbb{R}$ and
		$f_2: \mathbb{R}^{{N_n}-n} \to \mathbb{R}$ defined by
		\beq\label{equn:f2}
			f_1(\vec{v}) = \phi_{\sigma^2 \mathbf{I}}(\mathbf{\Lambda}_{(1)}^{(1/2)}\vec{v}),
			\quad  \quad
			f_2(\vec{u}) = \det\left( \mathbf{M}_{\mathcal{I}}\left( \mathbf{\Lambda}^{(1/2)}_{(2)} \vec{u}\right)\right)
			\vec V_{\left\{N_n - n\right\}}\left(\mathbf{\Lambda}_{(2)}^{(1/2)}\vec{u}\right).
		\eeq

		Then, $f_1 \in
		L^2(\mathbb{R}^n, \prod_{i=1}^{n}\phi(\vec u)d\vec u),~ \vec u \in \mathbb{R}^n$
		and $f_2 \in L^2(\mathbb{R}^{{N_n}-n}, \prod_{i=1}^{{N_n}-n}\phi(\vec u)d\vec u),~ \vec u \in
		\mathbb{R}^{{N_n}-n}$.
		Write the Hermite expansions for $f_1, f_2$ :
		\beqq
			f_1(\vec{u}) &=& \sum_{a_1,\cdots, a_n}
			d^{\sigma}_{a_1\cdots a_n}
			\prod_{i=1}^{n}H_{a_i}(u_i),\\
			f_2(\vec{v}) &=& \sum_{a_{n+1},\cdots, a_{{N_n}}}
			d_{a_{n+1}\cdots a_{{N_n}}}\prod_{i=1}^{{N_n}-n}H_{a_{n+i}}(v_i).
		\eeqq
		where
		\beqq
			d^{\sigma}_{a_1\cdots a_n} =
			\frac{1}{a_1!\cdots a_n!}
			\int_{\mathbb{R}^n}\phi_{\sigma^2 \mathbf{I}}(\mathbf{\Lambda}_{(1)}^{(1/2)}\vec{v})
			\prod_{i=1}^{n} H_{a_i}(v_i)\phi(v_i)d\vec{v},
		\eeqq
		and
		\begin{align*}
			d_{a_{n+1}\cdots a_{{N_n}}} =
			\frac{1}{a_{n+1}!\cdots a_{{N_n}}!}
			\int_{\mathbb{R}^{({N_n}-n)}}
			\det\left( \mathbf{M}_{\mathcal{I}}\left( \mathbf{\Lambda}^{(1/2)}_{(2)} \vec{u}\right)\right)
			\vec V_{\left\{N_n - n\right\}}&\left(\mathbf{\Lambda}_{(2)}^{(1/2)}\vec{u}\right)\\
			&\times\prod_{i={n+1}}^{{N_n}} H_{a_{i}}(u_i)\phi(u_i)d\vec{u}.
		\end{align*}

		Then, with probability one, we have
		\beqq
			f_1(\vec{Y}_{(1)})
			f_2(\vec{Y}_{(2)})
			= \sum_{a_1,\cdots, a_{{N_n}}}
			d^{\sigma}_{a_1\cdots a_n}
			d_{a_{n+1}\cdots a_{{N_n}}}
			\prod_{i=1}^{{N_n}}H_{a_i}(Y_i).
		\eeqq
		Re-arranging the sum gives
		\beqq
			f_1(\vec{Y}_{(1)})
			f_2(\vec{Y}_{(2)})
			&=&  \sum_{q=0}^{\infty}
			\sum_{\{a_i\} \in \pi_n(q)}
			d^{\sigma}_{a_1\cdots a_n}
			d_{a_{n+1}\cdots a_{{N_n}}}
			\prod_{i=1}^{{N_n}}H_{a_i}(Y_i)\\
%
			&=&
			\sum_{q=0}^{\infty}
			\sum_{\vec{a} \in \pi_n(q)}
			d^{\sigma}_{\vec{a}}
			\widetilde{H}_{\vec{a}}(\vec{Y}_{\vec s})d\vec s,
		\eeqq
		yielding
		\beqq
			F^{\sigma}_{(0,m)^n}[X] \overset{a.s.}{=}\int_{(0,m)^n}
			\sum_{q=0}^{\infty}
			\sum_{\vec{a} \in \pi_n(q)}
			d^{\sigma}_{\vec{a}}
			\widetilde{H}_{\vec{a}}(\vec{Y}_{\vec s})d\vec s.
		\eeqq
		To deduce the $L^2$ equality, write
		\beqq
			A_Q = \int_{(0,m)^n}
			\sum_{q=0}^{Q}
			\sum_{\vec{a} \in \pi_n(q)}
			d^{\sigma}_{\vec{a}}
			\widetilde{H}_{\vec{a}}(\vec{Y}_{\vec s})d\vec s.
		\eeqq
		Now note that the sequence $\{A_Q\}_{Q=1}^{\infty}$ is Cauchy. To prove this, note first that 
				\beqq
			\|A_{Q_1} -A_{Q_2}\|^2 &=&\E \left[\int_{(0,m)^n}
			\sum_{q=Q_1}^{Q_2}
			\sum_{\pi_n(q)}
			d^{\sigma}_{a_1\cdots a_n}
			d_{a_{n+1}\cdots a_{{N_n}}}
			\widetilde{H}_{\vec{a}}(\vec{Y}_{\vec s})d\vec s
			\right]^2 \\
			&\leq& m^n \int_{(0,m)^n}\E
			\left[
			\sum_{q=Q_1}^{Q_2}
			\sum_{\pi_n(q)}
			d^{\sigma}_{a_1\cdots a_n}
			d_{a_{n+1}\cdots a_{{N_n}}}
			\widetilde{H}_{\vec{a}}(\vec{Y}_{\vec s})d\vec s
			\right]^2  \\
					&\leq& m^n \int_{(0,m)^n}
			\sum_{q=Q_1}^{Q_2}\E
			\left[
			\sum_{\pi_n(q)}
			d^{\sigma}_{a_1\cdots a_n}
			d_{a_{n+1}\cdots a_{{N_n}}}
			\widetilde{H}_{\vec{a}}(\vec{Y}_{\vec s})d\vec s
			\right]^2.
		\eeqq
			where the last inequality here follows from the orthogonality of spaces $\mathcal{H}_q$.	
			
		Exploiting  the independence of the components of $\vec{Y}$, and
	 applying a generalized Mehler's formula (see proof of proposition 2.1 in \cite{Jose2014} and Lemma 10.7 in \cite{Az2009}), we can bound the above expression by
		\beqq       
			m^{2n}
			\sum_{q=Q_1}^{Q_2}
			\sum_{\pi_n(q)}
			d^{2,{\sigma}}_{a_1\cdots a_{{N_n}}}
			a_1!\cdots a_{{N_n}}!\ .
		\eeqq
		By convergence of the coefficients of the Hermite expansion,  the above tends to zero when $Q_1, Q_2$ increase, and so we have that $\{A_Q\}_{Q=1}^{\infty}$ is Cauchy.

				For the other faces, $J \in \{J\} \backslash T^{\circ}_n$ of dimension $d \equiv$ dim$(J)<n$,  we have
		a slightly different expression for the coefficients $d^{\sigma}_{\vec{a}}$ in \eqref{equn:PhiExp1}.
		Recall  \eqref{equn:GApprox}, from which it follows, similarly to the above, that the corresponding integrands are
		given by
		\beq\label{equn:f1Face}
			f_1(\vec{v}) =
			\phi_{\sigma^2 \mathbf{I}_{n\times n}}\left(\vec V_{\mathcal{I}_{J}}\left(\mathbf{\Lambda}_{(1)}^{(1/2)}\vec{v}\right)\right)
			\mathbbm{1}_{\left\{ \left\langle  \vec V_{\mathcal{I}_{\perp J}} \left(\mathbf{\Lambda}_{(1)}^{(1/2)}\vec{v}\right), \widetilde{\eta}_{J}\right\rangle \geq 0 \right\}},
		\eeq
		\beq\label{equn:f2Face}
			f_2(\vec{u}) = \det\left(   \mathbf{M}_{\mathcal{I}_{J}}\left( \mathbf{\Lambda}^{(1/2)}_{(2)} \vec{u}\right)\right)
			\vec  V_{\left\{N_n - n\right\}} \left(\mathbf{\Lambda}_{(2)}^{(1/2)}\vec{Y}_{(2)}\right).
		\eeq
		Thus, in terms of $\vec Y$    we have that
		\begin{align}\label{equn:coefficientslongOtherFace}
			d^{\sigma}_{\vec{a}|J} =
			\frac{1}{\vec{a}!}
			&\int_{\mathbb{R}^n}
			\phi_{\sigma^2 \mathbf{I}_{k\times k}}\left(\vec V_{\mathcal{I}_{J}}\left(\mathbf{\Lambda}_{(1)}^{(1/2)}\vec{v}\right)\right)
			\mathbbm{1}_{\left\{ \left\langle  \vec V_{\mathcal{I}_{\perp J}} \left(\mathbf{\Lambda}_{(1)}^{(1/2)}\vec{v}\right), \widetilde{\eta}_{J}\right\rangle \geq 0 \right\}}
			\prod_{i=1}^{n} H_{a_i}(v_i)\phi(v_i)d \vec v \notag \\
			~   \times
			&\int_{\mathbb{R}^{(N_k -k)}}
			\det\left(  \mathbf{M}_{\mathcal{I}_{J}}\left( \mathbf{\Lambda}^{(1/2)}_{(2)} \vec{u}\right)\right)
			\vec  V_{\left\{N_n - n\right\}} \left(\mathbf{\Lambda}_{(2)}^{(1/2)}\vec u\right)
			\prod_{i=1}^{N_k -k} H_{a_{n+i}}(u_i)\phi(u_i)
			d \vec u,
		\end{align}
		where $\vec{a}! = {a_1}!\cdots{a}_{n+\frac{1}{2}k(k+1)}!$,  and the second integral is evaluated 
		over the $N_k -k$ coordinates of $\vec u$
		which appear in $\mathbf{M}_{\mathcal{I}_{J}}\left( \mathbf{\Lambda}^{(1/2)}_{(2)} \vec{u}\right)$  and the one coordinate that appears
		in  $V_{\left\{N_n - n\right\}} \left(\mathbf{\Lambda}_{(2)}^{(1/2)}\vec u\right)$.
	\end{proof}

	\begin{lemma}\label{lemma:lemma3}
		$F_{(0,m)^n}[X]$
		admits the Wiener chaos expansion
		\beq\label{equn:PhiExp}
			&F_{(0,m)^n}[X] \ \eqinltwo\
			\sum_{q=1}^{\infty}
			\sum_{\pi_n(q)}
			d_{\vec{a}}\int_{(0,m)^n}{
			\widetilde{H}_{\vec{a}}(\vec{Y}_{\vec s})d\vec s}.
		\eeq
		Alternatively,
		\beq
			F_{(0,m)^n}[X] \eqinltwo \sum_{q=1}^{\infty}I_q(f_q^m).
		\eeq
		The variance, $\sigma^2_m$, of $F_{(0,m)^n}[X]$ is given by
		\beq
			\sigma^2_m =
			\sum_{q=1}^{\infty}
			\sum_{\vec{a} \in \pi_n(q)}
			\sum_{\vec{b} \in \pi_n(q)}
			d_{\vec{a}}\
			d_{\vec{b}}
			\ \vec{a}!\ \vec{b}!\ R^m(\vec{a},\vec{b}),
		\eeq
		where $f_q^m \in \mathfrak{H}^{\odot q} $ is given by 
		\beqq
			f_q^m
			= \sum_{\vec{a} \in \pi_n(q)}
			d_{\vec{a}}\int_{(0,m)^n}
			\textnormal{symm}\left(\varphi_{\vec s,1}^{\otimes a_1}\otimes\cdots\otimes\varphi_{\vec s,{N_n}}^{\otimes a_{{N_n}}}\right)d\vec s,
		\eeqq
		and the various coefficients are as follows:			
		\begin{align}     \label{equn:for_mean_calc}
			d_{\vec{a}} =&
			\frac{|\det({\mathbf{\Lambda}^{(1/2)}_{(1)}})|^{-(1/2)}}{\vec{a}!(2\pi)^{{n/2}}}
			\prod_{i=1}^{n} H_{a_i}(0) \notag \\
			&\times
			\int_{\mathbb{R}^{({N_n}-n)}}
			{
			\det\left( \mathbf{M}_{\mathcal{I}}\left( \mathbf{\Lambda}^{(1/2)}_{(2)} \vec{u}\right)\right)
			\vec V_{\left\{N_n + n\right\}}\left(\mathbf{\Lambda}_{(2)}^{(1/2)}\vec{u}\right)
			\prod_{i=1}^{{N_n}-n} H_{a_{n+i}}(u_i)\phi(u_i)d\vec u
			},
		\end{align}
		
		\beq
			R^m(\vec{a},\vec{b}) = m^n
			\int_{(-m,m)^n}
			\sum_{\substack{
			d_{ij}\geq 0 \\ \sum_i d_{ij}=a_j\\ \sum_j d_{ij}=b_i
			}
			}\vec{a}!\vec{b}!\prod_{1 \leq i,j \leq {N_n}} \frac{({K}_{ij}(\vec s))^{d_{ij}}}{(d_{ij})!}
			\prod_{1\leq k\leq n}\left(1-\frac{|\vec s_k|}{m}\right)
			d \vec s.
		\eeq
		and 
			\beq
			\sigma^2_m = \sum_{q=1}^{\infty}u_q^m, \quad\text{and}\quad
			u_q^m = \sum_{\vec{a} \in \pi_n(q)}
			\sum_{\vec{b} \in \pi_n(q)}
			d_{\vec{a}} \
			d_{\vec{b}}
			\ \vec{a}!\ \vec{b}!\ R^m(\vec{a},\vec{b}),
		\eeq	\end{lemma} 
	
	\begin{proof}
		By Lemma \ref{lemma:lemma1} and Lemma \ref{lemma:lemma2},
		it suffices to establish the $L^2$ convergence
		\beqq
			\sum_{q=0}^{\infty}
			\sum_{\vec{a} \in \pi_n(q)}
			d^{\sigma}_{\vec{a}}
			\int_{(0,m)^n}
			\widetilde{H}_{\vec{a}}(\vec{Y}_{\vec s})d\vec s
			\underset{\sigma \to 0}{\longrightarrow}
			\sum_{q=0}^{\infty}
			\sum_{\vec{a} \in \pi_n(q)}
			d_{\vec{a}}
			\int_{(0,m)^n}
			\widetilde{H}_{\vec{a}}(\vec{Y}_{\vec s})d\vec s.
		\eeqq
		It is straightforward that	$\lim\limits_{\sigma \to 0} d_{\vec{a}}^{\sigma}  \to d_{\vec{a}}$,
		where 

		\begin{align*}
			d_{\vec{a}} =&
			\frac{|\det({\mathbf{\Lambda}^{(1/2)}_{(1)}})|^{-(1/2)}}{\vec{a}!(2\pi)^{{n/2}}}
			\prod_{i=1}^{n} H_{a_i}(0)\\
			&\times
			\int_{\mathbb{R}^{({N_n}-n)}}
			{                
			\det\left( \mathbf{M}_{\mathcal{I}}\left( \mathbf{\Lambda}^{(1/2)}_{(2)} \vec{u}\right)\right)
			\vec V_{\left\{N_n - n\right\}}\left(\mathbf{\Lambda}_{(2)}^{(1/2)}\vec{u}\right)
			\prod_{i=1}^{{N_n}-n} H_{a_{n+i}}(u_i)\phi(u_i)d\vec u
			}.
		\end{align*}

		We start by showing that \eqref{equn:PhiExp} is in
		$L^2$. By Fatou's inequality
		\beqq
			\E\bigg{[}
				\sum_{q=0}^{Q}
				\sum_{\vec{a} \in \pi_n(q)}
				d_{\vec{a}}\int_{(0,m)^n}
				\widetilde{H}_{\vec{a}}(\vec{Y}_{\vec s})d\vec s
				\bigg{]}^2 &\leq&
				\varliminf\limits_{\sigma \to 0}
				\E \bigg{[}
				\sum_{q=0}^{Q}
				\sum_{\vec{a} \in \pi_n(q)}
				d^{\sigma}_{\vec{a}}\int_{(0,m)^n}
				\widetilde{H}_{\vec{a}}(\vec{Y}_{\vec s})d\vec s
			\bigg{]}^2\\
			&=&
			\varliminf\limits_{\sigma \to 0}
			\E
			\sum_{q=0}^{Q}
			\bigg{[}
			\sum_{\vec{a} \in \pi_n(q)}
			d^{\sigma}_{\vec{a}}\int_{(0,m)^n}
			\widetilde{H}_{\vec{a}}(\vec{Y}_{\vec s})d\vec s
			\bigg{]}^2,
		\eeqq
		the last line following from orthogonality.
		
		Adding some positive terms to the sum and then using Lemma \ref{lemma:lemma1}, the above is bounded by
			\beqq
			\varliminf\limits_{\sigma \to 0}
			\E
			\sum_{q=0}^{\infty}
			\bigg{[}
			\sum_{\vec{a} \in \pi_n(q)}
			d^{\sigma}_{\vec{a}}\int_{(0,m)^n}
			\widetilde{H}_{\vec{a}}(\vec{Y}_{\vec s})d\vec s
			\bigg{]}^2  = \E[F_{(0,m)^n}[X]]^2 < \infty.
		\eeqq
		We introduce yet another shorthand notation, to be used for the remaining part of the current proof. 
		\beqq
			\widetilde{I}^{\sigma}_{q} \definedas
			\sum_{\vec{a} \in \pi_n(q)}
			d^{\sigma}_{\vec{a}}\int_{(0,m)^n}
			\widetilde{H}_{\vec{a}}(\vec{Y}_{\vec s})d\vec s \quad\text{and} \quad
			\widetilde{I}_{q} \definedas
			\sum_{\vec{a} \in \pi_n(q)}
			d_{\vec{a}}\int_{(0,m)^n}
			\widetilde{H}_{\vec{a}}(\vec{Y}_{\vec s})d\vec s.
		\eeqq
		With the above notation we have
		\beqq
			F^{\sigma}_{(0,m)^n}[X] = \sum_{q=0}^{\infty}
			\widetilde{I}^{\sigma}_{q}
			\quad\text{and} \quad
			F_{(0,m)^n}[X] = \sum_{q=0}^{\infty}
			\widetilde{I}_{q}.
		\eeqq
		Since
		$\lim\limits_{\sigma \to 0} d_{\vec{a}}^{\sigma} \to d_{\vec{a}}$, we have, for a fixed $Q$,
		\beqq
			\lim_{\sigma \to 0}
			\bigg{\|}\sum_{q=0}^{Q}
			\widetilde{I}^{\sigma}_{q}
			\bigg{\|}_{L^2}
			=
			\bigg{\|}\sum_{q=0}^{Q}
			\widetilde{I}_{q}
			\bigg{\|}_{L^2}
			=
			\sum_{q=0}^{Q}
			\big{\|}
			\widetilde{I}_{q}
			\big{\|}_{L^2}.
		\eeqq
		Moreover, since $F_{(0,m)^n}[X], F_{(0,m)^n}^{\sigma}[X]\in L^2(\Omega)$,
		we have
		\beq\label{equn:ConvTo}
			\big{\|}F_{(0,m)^n}^{\sigma}[X]
			\big{\|}_{L^2}
			\underset{\sigma \to 0}{\to} \big{\|}
			F_{(0,m)^n}[X]  \big{\|}_{L^2}.
		\eeq
		Now,
		\begin{align*}
			\bigg{\|}
			\sum_{q=0}^{\infty}
			I_{q}
			- \sum_{q=0}^{\infty}
			I^{\sigma}_{q}
			\bigg{\|}_{L^2} \leq
			\bigg{\|}\sum_{q=0}^{Q}
			I_{q} -
			\sum_{q=0}^{Q}
			I^{\sigma}_{q}
			\bigg{\|}_{L^2}
			+
			\bigg{\|}
			\sum_{q=Q+1}^{\infty}
			I_{q}
			\bigg{\|}_{L^2}
			+
			\bigg{\|}
			\sum_{q=Q+1}^{\infty}
			I^{\sigma}_{q}
			\bigg{\|}_{L^2}.
		\end{align*}
		Given $\varepsilon>0$, we first choose $Q'$ sufficiently large so that
		\beq
			\bigg{\|}\sum_{q=Q'+1}^{\infty}
			I_{q}\bigg{\|}_{L^2} < \varepsilon/3 .
		\eeq
		Consequently, because of \eqref{equn:ConvTo}, we can then choose $\sigma$ sufficiently small  so that 
		\beq
			\bigg{\|}\sum_{q=Q'+1}^{\infty}
			I^{\sigma}_{q}\bigg{\|}_{L^2}
			< \varepsilon/3
			\quad \text{and}\quad
			\bigg{\|}\sum_{q=0}^{Q'}
			I_{q} -
			\sum_{q=0}^{Q'}
			I^{\sigma}_{q}
			\bigg{\|}_{L^2}
			\leq \varepsilon/3.
		\eeq
		Since $\|Y_i\| = 1$, we relay on the fundamental relation for the Wiener chaos 
		\beq
			\prod_{i=1}^{{N_n}} H_{a_i}(Y_i(\vec s))
			=I_{q}\left(\text{symm}(\varphi_{\vec s,1}^{\otimes a_1}\otimes
			\cdots\otimes\varphi_{\vec s,{N_n}}^{\otimes a_{{N_n}}})\right),
		\eeq
		to write the expansion for $F_{(0,m)^n}[X]$, and
		then apply Fubini's theorem for multiple Wiener integrals to arrive at
		\begin{align}\label{equn:fullcalc}
			F_{(0,m)^n}[X] &= \sum_{q=0}^{\infty}
			\sum_{\pi_n(q)}
			d_{a_1\cdots a_{{N_n}}}\int_{(0,m)^n}
			{\prod_{i=1}^{{N_n}} H_{a_i}(Y_i(\vec s))d\vec s }\notag \\
			&=\sum_{q=0}^{\infty}{I_{q}\left(
			\sum_{\vec{a} \in \pi_n(q)}
			d_{\vec{a}}\int_{(0,m)^n}
			\textnormal{symm}\left(\varphi_{\vec s,1}^{\otimes a_1}\otimes
			\cdots\otimes\varphi_{\vec s,{N_n}}^{\otimes a_{{N_n}}}\right)d\vec s \right)} \notag \\
			&= \sum_{q=1}^{\infty}I_q(f_q^m),
		\end{align}
		with
		\beqq
			f_q^m\left(\lambda_1,\cdots,\lambda_{q}\right)
			=
			\sum_{\vec{a} \in \pi_n(q)}
			d_{\vec{a}}\int_{(0,m)^n}
			\textnormal{symm}\left(\varphi_{\vec s,1}^{\otimes a_1}\otimes
			\cdots\otimes\varphi_{\vec s,{N_n}}^{\otimes a_{{N_n}}}\right)d\vec s,
		\eeqq
		where $\lambda_1,\ldots \lambda_q \in \mathbb{M}$, and $\mathbb{M}$ is as in \eqref{equn:space_M}.
		
		Next, we proceed to calculate the variance,
		$\sigma_m^2$. Using \eqref{equn:PhiExp} and the orthogonality of $\mathcal{H}_q$ for different $q$, we have
		\begin{align}
			\sigma_m^2 = \,
			&\E\left[
			\sum_{q=0}^{\infty}
			\sum_{\vec{a} \in \pi_n(q)}
			d_{\vec{a}}
			\int_{\small{(0,m)^n}}
			\widetilde{H}_{\vec{a}}(\vec{Y}_{\vec s})d\vec s
			\right]^2
			\notag  \\
			&=\sum_{q=0}^{\infty}
			\sum_{\vec{a} \in \pi_n(q)}
			\sum_{\vec{b} \in \pi_n(q)}
			d_{\vec{a}}
			d_{\vec{b}}
			\int \int_{(0,m)^{2n}}
			\E\left[\widetilde{H}_{\vec{a}}(\vec{Y}_{\vec s})
			\widetilde{H}_{\vec{b}}(\vec{Y}_{\vec u})
			\right]d\vec sd\vec u.
		\end{align}
		By stationarity, this equals
		\beqq
			\sum_{q=0}^{\infty}
			\sum_{\vec{a} \in \pi_n(q)}
			\sum_{\vec{b} \in \pi_n(q)}
			d_{\vec{a}}
			d_{\vec{b}}
			\int \int_{(0,m)^{2n}}
			\E\left[\prod_{i=1}^{{N_n}} H_{a_i}(Y_i(0))
			\prod_{i=1}^{{N_n}} H_{b_i}(Y_i(\vec u-\vec s))
			\right]d\vec sd\vec u.
		\eeqq
		Then, a change in variables leads to
		\beqq
			m^n\sum_{q=0}^{\infty}
			\sum_{\vec{a} \in \pi_n(q)}
			\sum_{\vec{b} \in \pi_n(q)}
			d_{\vec{a}}
			d_{\vec{b}}
			\int_{(-m,m)^n}
			\E\left[\prod_{i=1}^{{N_n}} H_{a_i}(Y_i(0))
			\prod_{i=1}^{{N_n}} H_{b_i}(Y_i(\vec \nu))\right]
			\prod_{1\leq k\leq n}\left(1-\frac{|\nu_k|}{m}\right)
			d\vec \nu.
		\eeqq
		When $|\vec{a}| = |\vec{b}|$, we have [see Proposition 2.2.1 in \cite{Nourdin2012}]
		\beqq \label{equn:innerpart}
			\E[\prod_{i=1}^{{N_n}} H_{a_i}(Y_i(0))
			\prod_{i=1}^{{N_n}} H_{b_i}(Y_i(\nu))]
			= \vec{a}!\vec{b}!
			\sum_{\substack{
			d_{ij}\ \geq \ 0 \\ \sum_i d_{ij}\ =\ a_j\\ \sum_j d_{ij}\ =\ b_i
			}
			}
			\prod_{1 \leq i, j \leq {N_n}} \frac{({K}_{ij}(\vec \nu))^{d_{ij}}}{(d_{ij})!},
		\eeqq
		and zero otherwise. 
		Thus, writing
		\beqq
			R^m(\vec{a},\vec{b}) = {m^n}
			\int_{(-m,m)^n}\vec{a}!\vec{b}!
			\sum_{\substack{
			d_{ij}\ \geq \ 0 \\ \sum_i d_{ij}\ =\ a_j\\ \sum_j d_{ij}\ =\ b_i
			}
			}\prod_{1 \leq i,j \leq {N_n}} \frac{({K}_{ij}(\vec \nu))^{d_{ij}}}{(d_{ij})!}
			\prod_{1\leq k\leq n}\left(1-\frac{|\nu_k|}{m}\right)
			d\vec \nu,
		\eeqq
		the variance is given by
		\beqq
			\sigma_m^2=
			\sum_{q=1}^{\infty}
			\sum_{\vec{a} \in \pi_n(q)}
			\sum_{\vec{b} \in \pi_n(q)}
			d_{\vec{a}}
			d_{\vec{b}}\
			\vec{a}!\
			\vec{b}!\
			R^m(\vec{a},\vec{b})
			= \sum_{q=1}^{\infty}u_q^m,
		\eeqq
		where $u_q^m\geq 0$ is given by
		\beqq
			u_q^m =   \sum_{\vec{a} \in \pi_n(q)}
			\sum_{\vec{b} \in \pi_n(q)}
			d_{\vec{a}}\
			d_{\vec{b}} \
			\vec{a}!\
			\vec{b}!\
			R^m(\vec{a},\vec{b}),
		\eeqq
		and we are done.
	\end{proof}


	\begin{lemma}\label{lemma:lemma4}
		The coefficients in \eqref{equn:PhiExp} satisfy
		\beq\label{equn:CondNorm}
			\sum\limits_{\vec{a}\in \pi_n(q)}
			d^2_{\vec{a}}\vec{a}!
			< C q^n.
		\eeq
	\end{lemma}

	\begin{proof}
		The coefficients are
		\beqq
			d_{a_1\ \cdots \ a_n} =
			\frac{|{\det}({\mathbf{\Lambda}^{(1/2)}_{(1)}})|^{-(1/2)}}{(2\pi)^{{n/2}}a_1!\cdots a_n!}\prod_{i=1}^{n}H_{a_i}(0),
		\eeqq
		\beqq
			d_{a_{n+1}\ \cdots\ a_{{N_n}}} = \frac{1}{a_{n+1}!\cdots a_{{N_n}}!}
			\int_{R^{{N_n}-n}}f_2(\vec s)
			\prod_{i=1}^{{N_n}-n}H_{a_{n+i}}(s_i)\phi{(s_i)}d\vec s,
		\eeqq
		with $f_2$ as defined in \eqref{equn:f2}.
		It is straightforward to see that $f_2(\vec s)$ is a polynomial
		of degree ${N_n}- n + 1$ and thus has a finite Hermite polynomial
		expansion. This means that all the terms
		$d_{a_{n+1}\cdots a_{N_n}}$ with any of the indexes
		$a_i > {N_n}-n+1$, $ i \in \{(n+1), \ldots , {N_n}\}$ are zero.
		Setting
		\beqq
		C \definedas |{\det}({\mathbf{\Lambda}^{(1/2)}_{(1)}})|^{-(1/2)} \times \max_{a_{n+1},\ldots a_{{N_n}}}{(d^2_{a_{n+1}\cdots a_{N_n}}a_{n+1}!\cdots a_{N_n}!)},
		\eeqq
		gives
		\beqq
			\sum_{\pi_n(q)}d^2_{a_1 +\ \cdots \ + a_{{N_n}}}a_1!\cdots a_{{N_n}}!\leq C \sum\limits_{a_1,\ldots,a_n \in \pi_n(q)}
			\bigg{(}
			\frac{H_{a_i}(0)\cdots H_{a_n}(0)}{(2\pi)^{{n/2}}a_1!\cdots a_n!}
			\bigg{)}^2a_1! \cdots a_n!.
		\eeqq
		Using \cite{Imkeller95}, Proposition 3, we have $\left|\frac{(H_{a_i}(0))^2}{{a_i}!}\right| \leq C$
		and thus
		\beqq
			\sum_{\pi_n(q)}d^2_{a_1 +\ \cdots \ + a_{{N_n}}}a_1!\cdots a_{{N_n}}!\leq  C \sum\limits_{a_1,\ldots,a_n \in\pi_n(q)}{1} < C q^n.
		\eeqq
		\normalfont
		For other faces $J \in \{J\}\backslash T^{\circ}_n$, by \eqref{equn:GApprox} the 
		contribution to the Euler integral of face $J$ is given by the limits of expressions of the form 
		\beqq
			\int_{J}
			\phi_{\sigma^2 \mathbf{I}}{(\nabla{{X}_{|J}(\vec s)})\det(\nabla^2{X}_{|J}(\vec s))X(\vec s)				
			\mathbbm{1}_{\{ \langle \nabla X(\vec{s}), \vec{\eta}_J\rangle \geq 0\}}
			d\vec s },					
		\eeqq
		which are bounded by
		\beqq
			\int_{J}
			\phi_{\sigma^2 \mathbf{I}}{(\nabla{{X}_{|J}(\vec s)})\left(1+\left(\det(\nabla^2{X}_{|J}(\vec s))\right)^2\right)\left(1+ (X_{|J}(\vec{s}))^2\right)d\vec s }.
		\eeqq
		Although, the functions in the bound are slightly different to the corresponding functions in the previous development for the contribution of $T_n^{\circ}$, the remainder of the argument is essentially the same, and so we shall not write out the details.
	\end{proof}

	\subsection{Proof of  Theorem \ref{theorem:main}}
	   
   		As previously shown in Lemma \ref{lemma:lemma1}, the sum corresponding to the interior critical points
		can be written as the limit of
		\beq\label{equn:TheExpression}
			F_{(0,m)^n}^{\sigma}[X] =
			\int_{(0,m)^n}
			{\phi_{\sigma^2 \mathbf{I}}(\nabla{X}(\vec s))\det(\nabla^2{X}(\vec s))X(\vec s)d\vec s}.
		\eeq
		By Lemmas  \ref{lemma:lemma1} and  \ref{lemma:lemma3} we have the Wiener chaos representation
		of \eqref{equn:TheExpression}
		\begin{align}
			F_{(0,m)^n}[X] = \sum_{q=1}^{\infty}I_q(f_q^m).
		\end{align}
		We, therefore, need to establish a CLT for  
		\beqq
			\frac{1}{m^{n/2}}\sum_{q=1}^{\infty}I_q(f_q^m).
		\eeqq
		The CLT will follow immediately  from  Theorem
		\ref{theorem:CLTTheoremFirst} once we have checked that all the conditions of the theorem hold in our case.      		
	
		We start with Condition {\it{(}c\it{)}}  of Theorem
		\ref{theorem:CLTTheoremFirst}. Conditions {\it{(}a\it{)}} and {\it{(}b\it{)}} will be deduced from our previous results and the
		proof of {\it{(}d\it{)}}. As far as {\it (c)} is concerned, note firstly that, by Lemma \ref{lemma:lemma3}, we
		have
		\beqq
			f_q^m\left(\lambda_1,\cdots,\lambda_{q}\right)
			=
			\sum_{\vec{a} \in \pi_n(q)}
			d_{\vec{a}}\int_{(0,m)^n}
			\textnormal{symm}\left(\varphi_{\vec s,1}^{\otimes a_1}\otimes\cdots\otimes\varphi_{\vec s,{N_n}}^{\otimes a_{{N_n}}}\right)d\vec s,
		\eeqq
		where
		\beqq
			\textnormal{symm}\left(\varphi_{\vec s,1}^{\otimes a_1}\otimes\cdots\otimes\varphi_{\vec s,{N_n}}^{\otimes a_{{N_n}}}\right)=
			\frac{1}{q!}\sum_{\sigma}\varphi_{\vec s,1}^{\otimes a_1}\otimes\cdots\otimes\varphi_{\vec s,{N_n}}^{\otimes a_{{N_n}}}\left(\lambda_{\sigma(1)},\cdots,\lambda_{\sigma(q)}\right).
		\eeqq
		Since $|\vec{a}| = q$, the inner sum can be written as
		\beq\label{equn:SumOfPhi}
			\sum_{j_1,\ldots,j_q=1}^{{N_n}} c_{j_1,\ldots,j_q}\varphi_{s,j_1}\otimes\cdots\otimes\varphi_{s,j_q}
		\eeq
		with the appropriate coefficients
		$\{c_{j_1,\ldots,j_q}\}^{{N_n}}_{j_1,\ldots,j_q=1}$, such that 
		$c_{j_1,\ldots,j_q} = 0$, whenever $\sum_i {j_i} \neq q$. Take
		$C(q) = \max_{j_1,\ldots,j_q}\{c_{j_1,\ldots,j_q}\}$. Using the
		above notation, we write
		\begin{multline*}
			\frac{1}{m^{n}}(f_q^m\otimes_{r}f_q^m) =
			\frac{1}{m^{n}}
			\sum_{\vec{a} \in \pi_n(q)}
			\sum_{\vec{b} \in \pi_n(q)}
			d_{\vec{a}}
			d_{\vec{b}} \\
			\times \int_{(0,m)^{2n}}
			\sum_{
			\substack{
			{j_1,\ldots,j_q}  \\
			{k_1,\ldots,k_q}
			}}
			c_{j_1,\ldots,j_q}
			c_{k_1,\ldots,k_q}
			\left(\varphi_{\vec s,j_1}\otimes\cdots\otimes\varphi_{\vec s,j_q}\right)
			\otimes_r
			\left(\varphi_{\vec u,k_1}\otimes\cdots\otimes\varphi_{\vec u,k_q}\right)
			d\vec sd\vec u.
		\end{multline*}
		The following is true (see Section \ref{subsection:Operations})
		\begin{multline*}
			\left(\varphi_{\vec x,j_1}\otimes\cdots\otimes\varphi_{\vec x,j_q}\right)
			\otimes_r
			\left(\varphi_{\vec y,k_1}\otimes\cdots\otimes\varphi_{\vec y,k_q}\right)
			\\
			= \left[ \prod_{l=1}^{r}\langle\varphi_{\vec x,j_l},\varphi_{\vec y,k_l}\rangle_{\mathfrak{H}}\right]
			\varphi_{\vec x,j_{r+1}} \otimes \cdots \varphi_{\vec x,j_{q}} 
			\otimes \varphi_{\vec y,k_{r+1}} \otimes \cdots \varphi_{
			\vec y,k_{q}} \\
			=\prod_{l=1}^{r}{K}_{j_i,k_l}(\vec x-\vec y)
			\varphi_{\vec x,j_{r+1}} \otimes \cdots \varphi_{
			\vec x,j_{q}} 
			\otimes \varphi_{\vec y,k_{r+1}} \otimes \cdots \varphi_{\vec y,k_{q}} ,
		\end{multline*}
		and
		\begin{multline*}
			\big{\langle}
			\varphi_{\vec x,j_{r+1}} \otimes \cdots \varphi_{\vec x,j_{q}} 
			\otimes \varphi_{\vec y,k_{r+1}} \otimes \cdots \varphi_{\vec y,k_{q}} 
			, 
			\varphi_{\vec w,i_{r+1}} \otimes \cdots \varphi_{\vec w,i_{q}} 
			\otimes \varphi_{\vec z,m_{r+1}} \otimes \cdots \varphi_{\vec z,m_{q}} 
			\big{\rangle}_{\mathfrak{H}^{\otimes2q}}\\
			= \prod_{l=1}^{q-r}{K}_{j_l,i_l}(\vec x-\vec w)\prod_{l=1}^{q-r}{K}_{k_l,m_l}(\vec y-\vec z).
		\end{multline*}
		We have
		\beqq
			\prod_{l=1}^{r}{K}_{j_i,k_l}(\vec x-\vec y) \leq \psi^{r}(\vec x-\vec y),
		\eeqq
		and 
		\beqq
			\prod_{l=1}^{q-r}{K}_{j_i,i_l}(\vec x-\vec w)\prod_{l=1}^{q-r}{K}_{k_l,m_l}(\vec y-\vec z) \leq
			\psi^{q-r}(\vec x- \vec w)\psi^{q-r}(\vec y-\vec z).
		\eeqq
		Since the number of summands in \eqref{equn:SumOfPhi} is less then $({N_n})^q$, we have that
				\begin{multline*}
			{m^{-n}}\|f_q^m\otimes_{r}f_q^m\|^{2}_{\mathfrak{H}^{\otimes2q}} = {m^{-n}}\langle f_q^m\otimes_{r}f_q^m, f_q^m\otimes_{r}f_q^m \rangle_{\mathfrak{H}} \leq
			{m^{-n}}
			\left(
			\sum_{\vec{a} \in \pi(q)} d^2_{\vec{a}}
			\right) 
			\left(({N_n})^q C(q)\right)^2\mathcal{Z}(t),
		\end{multline*}
		where
		\beqq
			\mathcal{Z}(m) = \int_{(0,m)^{4n}}\psi^{r}(\vec x-\vec y)\psi^{r}(\vec w-\vec z)\psi^{q-r}(\vec x-\vec w)\psi^{q-r}(\vec y-\vec z)d\vec xd\vec yd\vec wd\vec z.
		\eeqq
		Next, using $\psi^{r}(\vec x-\vec y)\psi^{q-r}(\vec x-\vec w)\leq \psi^{q}(\vec x-\vec y) + \psi^{q}(\vec x-\vec w)$ we have 
		\beqq
			\mathcal{Z}(m) \leq \mathcal{Z}_1(m) + \mathcal{Z}_2(m),
		\eeqq
		where
		\beqq
			\mathcal{Z}_1(m) = \int_{(0,m)^{4n}} \psi^r(\vec w-\vec z)\psi^q(\vec x-\vec w)\psi^{q-r}(\vec z-\vec y)d\vec xd\vec yd\vec wd\vec z,
		\eeqq
		\beqq
			\mathcal{Z}_2(m) = \int_{(0,m)^{4n}} \psi^r(\vec w-\vec z)\psi^q(\vec x-\vec y)\psi^{q-r}(\vec z-\vec y)d\vec xd\vec yd\vec wd\vec z.
		\eeqq
		Integrating with respect to $\vec x$ and using
		\beqq
			\int_{(0,m)^n} \psi^q(\vec x-\vec w)d\vec x \leq \int_{\mathbb{R}^n} \psi^q(\vec v)d\vec v <
			\infty,
		\eeqq
		then integrating with respect to the remaining coordinates and using  
		\begin{align*}
			\int_{(0,m)^{3n}} \psi^r(\vec w-\vec z)\psi^{q-r}(\vec z-\vec y)&* d\vec yd\vec wd\vec z =\\
			= \int_{(-m,m)^{2n}} &\psi^r\star \psi^{q-r}(\vec w-\vec y)d\vec wd\vec y 
			\leq \int_{\mathbb{R}^n} \psi^r \star \psi^{q-r}(\vec v)d\vec v
		\end{align*}
		we get 
		\beqq
			\left(
			\sum_{\vec{a} \in \pi(q)} d^2_{\vec{a}}
			\right)
			\left(({N_n})^q C\right)^2\mathcal{Z}(m) < \infty,
		\eeqq
		from which it follows that
		\beqq
			{m^{-n}}\|f_q^m\otimes_{q-p}f_q^m\|_{\mathfrak{H}^{\otimes2p}} \leq  \frac{C'}{m^n} .
		\eeqq
		Consequently
		\beqq
			{m^{-n}}\|f_q^m\otimes_{q-p}f_q^m\|_{\mathfrak{H}^{\otimes2p}} \underset{m \to \infty}{\longrightarrow} 0
		\eeqq Thus, we have established that condition {\it{(}}c{\it{)}} of Theorem \ref{theorem:CLTTheoremFirst} is satisfied.
		
		We now turn to the condition  {\it{(}d\it{)}}. We have to show that
		\beqq \label{equn:pointd}
			\sup_{m \geq 1} {m^{-n}}\sum_{q=Q+1}^{\infty}q!\|f^m_q\|^2_{\mathfrak{H}^{\otimes q}} \underset{Q \to \infty}{\longrightarrow} 0.
		\eeqq
		The expression
		${m^{-n}}\sum_{q=Q+1}^{\infty}q!\|f^m_q\|^2_{\mathfrak{H}^{\otimes
		q}}$ is the variance of the tail of the Wiener chaos expansion of
		${m^{-n}}F_{(0,m)^n}[X]$. We need to show that it converges
		uniformly to zero. We have already developed this expansion in the
		previous results, with the only difference that now we have a
		normalization of ${m^{-n}}$,
		\beq
		\label{mn:rob}
			{m^{-n}}\sigma^2_m(Q) &=& \sum_{q=Q}^{\infty}
			\sum_{\vec{a} \in \pi_n(q)}
			\sum_{\vec{b} \in \pi_n(q)}
			d_{\vec{a}}\
			d_{\vec{b}} \\  &&\qquad
			\times \int_{(-m,m)^n}
			\E\left[\prod_{i=1}^{{N_n}} H_{a_i}(Y_i(0))
			\prod_{i=1}^{{N_n}} H_{b_i}(Y_i(\vec \nu))\right]
			\prod_{1\leq k\leq n}(1-\frac{|\nu_k|}{m})
			d\vec \nu.\notag
					\eeq
		
		To show the uniform convergence 
		we write an upper bound $C(Q)$ which is independent of $m$ 
		and vanishes as $Q \to \infty$.
		To construct such a bound we use Lemma 1 in \cite{arcones1994}, which, for completeness, we reproduce.
		\begin{lemma}[\cite{arcones1994}]
				\label{lemma:arcones}
		Let $\vec V $ and $\vec W$ be two zero-mean Gaussian random vectors on $\mathbb{R}^d$, and assume that  $\E V_i V_j = \E W_i W_j = \delta_{ij}$.  Let $h:\mathbb{R}^d\to\mathbb{R}$ have
		 Hermite rank  $r$. (i.e.\ the lowest degree  polynomial appearing in its Hermite expansion has degree $r$.) 
		Write  $\psi_{*}$ for the supremum of the sum of 
		the rows or columns of the covariance matrix  $Cov(\vec V, \vec W)$, and 
		assume that $\psi_{*} < 1$. Then 
		\[
		\left|\E[h(\vec V)- \E h(\vec V)][h(\vec{W})- \E h(\vec{W})] \right| \leq \psi^{r}_{*}\E h^2(\vec V).
		\]
		\end{lemma}

		Returning  to the proof of Theorem  \ref{theorem:main}, we now apply this lemma with 
\beqq		
	\vec{V}=(Y_1(0),\dots,Y_{N_n}(0)),  \qquad \vec{W}=(Y_1(\vec\nu),\dots,Y_{N_n}(\vec\nu)),
\eeqq
and $h:\mathbb{R}^{N_n}\to\mathbb{R}$ given by		
	\beqq
	h(\vec x) =\prod_{i=1}^{{N_n}} H_{a_i}(x_i).
	\eeqq 
	It is easy to check that
		\beqq
		\psi_{*} \leq K^q\psi^q(\tau)
			\quad \text{and} \quad
			\E h^2 = \sum_{\pi_n(q)}d^2_{\vec{a}}\ \vec{a}!.
		\eeqq
		Furthermore, since $q>0$,  $\E \prod_{i=1}^{{N_n}} H_{a_i}(Y_i) = 0$,  and for $|\vec \tau|$ large enough,  by the assumption on $\psi(\vec \tau)$, we have that $K^q\psi^q(\tau) < 1$.
		
		We now choose arbitrary $s \in \mathbb{R}^{+}$ and split the integral over two domains 
		\begin{align}
			\int_{(-m,m)^n}
			\E&\left[\prod_{i=1}^{{N_n}} H_{a_i}(Y_i(0)) 
			\prod_{i=1}^{{N_n}} H_{b_i}(Y_i(\vec \nu))\right] \prod_{1\leq k\leq n}\left(1-\frac{|\nu_k|}{m}\right)  d\vec \nu 
					\notag       \\
			& = \int_{R^n_0(s)}
			\E\left[\prod_{i=1}^{{N_n}} H_{a_i}(Y_i(0)) 
			\prod_{i=1}^{{N_n}} H_{b_i}(Y_i(\vec \nu))\right] \prod_{1\leq k\leq n}\left(1-\frac{|\nu_k|}{m}\right) d\vec \nu~
			\notag 	      \\&
			\quad +
			\int_{(-m,m)^n\backslash R^n_0(s)}
			\E\left[\prod_{i=1}^{{N_n}} H_{a_i}(Y_i(0)) 
			\prod_{i=1}^{{N_n}} H_{b_i}(Y_i(\nu))\right] \prod_{1\leq k\leq n}\left(1-\frac{|\nu_k|}{m}\right) d\vec \nu.
		\end{align}
		Here $R^n_0(s)$ is $n$-dimensional cube of side length $s$, centered at the origin.
		
				For the first term corresponding to the integral over $R^n_0(s)$, we write 
		\beqq
			\sum_{q=Q}^{\infty} \int_{R^n_0(s)} \E\left[\sum_{\vec{a} \in \pi_n(q)}				      
			d_{\vec{a}}
			\prod_{i=1}^{{N_n}} H_{a_i}(Y_i(0)) 
			\sum_{\vec{b} \in \pi_n(q)}
			d_{\vec{b}}		      
			\prod_{i=1}^{{N_n}} H_{b_i}(Y_i(\vec \nu))\right]
			\prod_{1\leq k\leq n}\left(1-\frac{|\nu_k|}{m}\right)
			d\vec \nu.
		\eeqq
		 Lemma \ref{lemma:lemma3}  implies that the above sum is finite for all  $Q$, and so   it converges to zero as 
		 $Q\to\infty$. Regarding the second term, and reintroducing the summation from \eqref{mn:rob}, consider
		\begin{align}\label{equn:secondPart}
			\sum_{q=Q}^{\infty} 	\int_{(-m,m)^n \backslash R^n_0(s)} \E\Bigg{[}\sum_{\vec{a} \in \pi_n(q)}				      
			d_{\vec{a}}
			\prod_{i=1}^{{N_n}} H_{a_i}(Y_i(0)) 
			\sum_{\vec{b} \in \pi_n(q)}
			d_{\vec{b}}\notag  \prod_{i=1}^{{N_n}}& H_{b_i}(Y_i(\vec \nu))\Bigg{]} \\   
			&\times
			\prod_{1\leq k\leq n}\left(1-\frac{|\nu_k|}{m}\right)
			d\vec \nu.
		\end{align}
		By Lemma \ref{lemma:arcones}, 
		we have
		\beqq 
			\left| 
			\E\left[\sum_{\vec{a} \in \pi_n(q)}				      
			d_{\vec{a}}
			\prod_{i=1}^{{N_n}} H_{a_i}(Y_i(0)) 
			\sum_{\vec{b} \in \pi_n(q)}
			d_{\vec{b}}		      
			\prod_{i=1}^{{N_n}} H_{b_i}(Y_i(\vec \nu))\right]\right| 
			\leq K^q\psi^q(\vec \nu)
			\sum_{\pi_n(q)}d^2_{\vec{a}}\vec{a}!,
		\eeqq
		so we can bound the second integral by
		\begin{multline*}
			\sum_{q=Q}^{\infty} 	\int_{(-m,m)^n \backslash R^n_0(s)} \left| K^q\psi^q(\vec \nu)
			\sum_{\pi_n(q)}d^2_{\vec{a}}\vec{a}!\right|
			\prod_{1\leq k\leq n}\left(1-\frac{|\nu_k|}{m}\right)
			d\vec \nu	 \leq       		\\ \leq
			\sum_{q=Q}^{\infty} 	\int_{(-m,m)^n \backslash R^n_0(s)}K^q\psi^q(\vec \nu)
			\sum_{\pi_n(q)}d^2_{\vec{a}}\vec{a}!		      		  	    				    
			d\vec \nu.	   
		\end{multline*}
		By Lemma \ref{lemma:lemma4} we have $\sum\limits_{\vec{a}\in \pi_n(q)} d^2_{\vec{a}}\vec{a}! < C q^n $, 
		so that 
		\beqq		
			\int_{(-m,m)^n \backslash R^n_0(s)}K^q\psi^q(\vec \nu)
			\left(\sum_{\pi_n(q)}d^2_{\vec{a}}\vec{a}!\right)
			d\vec \nu	    \leq	 \int_{(-m,m)^n \backslash R^n_0(s)}K^q\psi^q(\vec \nu)
			C q^n		      		  	    				    
			d\vec \nu.	 		
		\eeqq
		Due to the assumption that $\psi(\vec \nu) \to 0$, we have,  for $s \in \mathbb{R}^{+}$ large enough, that 
		\beq 
			\psi(\vec \nu \in \mathbb{R}^n \backslash R^n_0(s)) < \varepsilon < \frac{1}{K},
		\eeq	
		which leads to the  bound
		\beqq
			\sum_{q=Q+1}^{\infty}
			C q^n	
			\frac{K^{q}\varepsilon^{q}}{\varepsilon}
			\int_{\mathbb{R}^n} \psi(\vec \nu)d\vec \nu.
		\eeqq
		Since $K \varepsilon < 1$ and 
		\beqq
			\int_{\mathbb{R}^n} \psi(\vec \nu) d\vec \nu < \infty,
		\eeqq
		the integral in \eqref{equn:secondPart} converges to zero uniformly in $m$.
		Overall, we conclude that ${m^{-n}}\sigma^2_m(Q) \underset{Q \to \infty}{\longrightarrow} 0$ uniformly in $m$, which establishes that Condition ${\it(d)}$ is satisfied.
		
		To show {\it(a)} and {\it(b)}, we note hat we have already
		demonstrated that for sufficiently large $s$ the integral over
		$(-m,m)^n \backslash R^n_0(s)$ converges to zero. Then,
		\beqq
			R^m(\vec{a},\vec{b}) \to R(\vec{a},\vec{b}) =
			\int_{\mathbb{R}^n}\vec{a}!\vec{b}!
			\sum_{\substack{
			d_{ij}\geq 0 \\ \sum_i d_{ij}=a_j\\ \sum_j d_{ij}=b_i
			}
			}\prod_{1 \leq i,j \leq {N_n}} \frac{({K}_{ij}(\vec \nu))^{d_{ij}}}{(d_{ij})!}
			d\vec \nu,
		\eeqq
		leading to
		\beq
		\label{sigmapsi:eqn}
			\sigma^2_m \underset{m\to \infty}{\longrightarrow} \sigma^2_\Psi  \definedas \sum_{q=1}^{\infty}u_q,
		\eeq
		where
		\beqq
			u_q = \sum_{\vec{a} \in \pi_n(q)}
			\sum_{\vec{b} \in \pi_n(q)}
			d_{\vec{a}}\
			d_{\vec{b}}
			\vec{a}!\vec{b}!R(\vec{a},\vec{b}).
		\eeqq
		This completes the proof of Theorem \ref{theorem:main}.
	\qed
		
\section{The mean value of upper Euler integral}\label{sec:meanvalue}
\label{mean:sec}
	In this section, we discuss the mean value of the Euler integral.
	As shown by \cite{Bobrowski2011}, the mean value of the Euler integral of Gaussian random field scales
	not by the volume of the domain of integration, as one would
	expect, but according to a one-dimensional measure of the
	domain. Specifically, 
	\beq\label{equn:mean}
		\E[\Psi_M[X]] = -\frac{\mathcal{L}^{X}_1(M)}{\sqrt{2\pi}}
	\eeq
	where $\mathcal{L}^{X}_1(M)$ is the first Lipschitz-Killing
	curvature  of $M$, as  evaluated with respect to the metric induced by the random field $X$. (cf.\ \cite{Adler2007} for definitions.)
	
	In this section, we re-establish this result by direct evaluation of
	the mean value through the Wiener chaos decomposition of the Euler
	integral. To do so, we make use of the next proposition, which can be proven using symmetry considerations.

	\begin{proposition}\label{prop:propDetX}
	Let X be tame Gaussian on $\mathbb{R}^n$, and $n > 1$. Then,
		\beq\label{equn:DetX}
			\E\left[\det(\nabla^2 X) X\right]=0.
		\eeq
	\end{proposition}
	
	Recall that in proving the CLT of the previous section we concentrated only on critical points in the interior of the parameter space which contributed to the Euler integral. Now, however, we need to consider all such points, since we are looking at an un-normalised mean, rather than an asymptotic limit.
	
	What is now interesting, and very different to what we saw before, is that 
	  the chaos approach shows that none of the faces of
	dimension different from one  can contribute to $\E[\Psi_M[X]]$, including the interior face. This gives, from this angle at least, some new intuition into the Bobrowski-Borman result. 
	
	To justify this  claim, note that in our chaos expansions the mean values of random variables in a chaos of order greater
	than zero vanish, so that possible contribution to  mean values may come only from the zeroth
	chaos. This, however, is characterized by the coefficients $d_{0\cdots 0|J}$. Let $\{d_{\vec{a}}\}_{\vec{a}}$ be the coefficients of
	the chaos of some face $J$. In
	general, each $d_{\vec{a}}$ in the Wiener chaos expansion for $J$ of dimension $k$ factorizes as $d^{(1)}_{a_{1}\cdots
	a_{k}} d^{(2)}_{a_{k+1}\cdots a_{N(k)}}$, where
	\begin{align}
		d^{(1)}_{_{a_{1}\cdots a_{k}}|J} = \lim_{\sigma \to 0}
		\frac{1}{\vec{a}!}
		\int_{\mathbb{R}^n}
		\phi_{\sigma^2 \mathbf{I}_{k\times k}}\left(\vec V_{\mathcal{I}_{J}}\left(\mathbf{\Lambda}_{(1)}^{(1/2)}\vec{v}\right)\right)
		\mathbbm{1}&_{\left\{ \left\langle  \vec V_{\mathcal{I}_{\perp J}}\left(\mathbf{\Lambda}_{(1)}^{(1/2)}\vec{v}\right), \widetilde{\eta}_{J}\right\rangle \geq 0 \right\}} \notag \\
		&\qquad \qquad \times \prod_{i=1}^{n} H_{a_i}(v_i)\phi(v_i)d \vec v      ,
	\end{align}
        
	\begin{align}
		d^{(2)}_{a_{k+1}\cdots a_{N_k}|J} =
		\int_{\mathbb{R}^{(N_k -k))}}
		\det\left(  \mathbf{M}_{\mathcal{I}_{J}}\left( \mathbf{\Lambda}^{(1/2)}_{(2)} \vec{u}\right)\right)
		\vec  V_{\left\{N_n - n\right\}} &\left(\mathbf{\Lambda}_{(2)}^{(1/2)}\vec{Y}_{(2)}\right) \notag \\
		\times
		&\prod_{i=1}^{N_k -k} H_{a_{n+i}}(u_i)\phi(u_i)
		d \vec u.
	\end{align}
                                        
	Note, we have that $d^{(2)}_{0\cdots 0|J} = \E [ \det(\nabla^2 X_{|J}) X_{|J}]$. Thus, by Proposition
	\ref{prop:propDetX}, the contribution to the mean value of the faces 
	$J,~\dim J>1$ is zero. Since the underlying field is centered it is obvious that the zero dimensional faces, the vertexes, do not contribute to the mean as well. Overall, the conclusion is that only the edges contribute to the mean value of Euler integral. 
	
	To see what this implies in a simple example, take $M = T_n $ with additional assumption of isotropy. Then
	\beq
		d^{(2)}_{0,\ldots,0} =
		\E[\frac{\partial^2 X(\vec{s})}{\partial s_is_i} X(\vec{s})] =  - \frac{\partial^2\rho(0)}{\partial \tau_i \partial\tau_i} = -\lambda_{2},
	\eeq
	where $\lambda_2$ is the second spectral moment of $X$.

	To calculate $d^{(1)}_{0,\ldots,0}$, we consider
	\beq
		\phi_{\sigma I}{(\nabla{{X}_{|J}(\vec{s}))})\det(\nabla^2{X}_{|J}(\vec{s}))X(\vec{s})
		\mathbbm{1}_{\{ \langle \nabla X(\vec{s}), \vec \eta_J\rangle \geq 0 \}}.
		}
	\eeq
	Having summed over all parallel edges and taking expectations, we can eliminate the indicator
	term $\mathbbm{1}_{\{\langle \nabla X(\vec{s}), \eta_J\rangle \geq 0 \}}$, since by stationary we translate everything
	to the same range over $(0,m)$ which yields
	\beq
		\sum_{\{J_i| J_i \text{ parallel to } J\}} \mathbbm{1}_{\{\langle \nabla X(\vec{s}), \vec \eta_{J_i}\rangle \geq 0\}} \equiv 1.
    \eeq
	Consequently, by \eqref{equn:for_mean_calc},
	\beq
		d^{(1)}_{0} =
		\frac{|\lambda_2|^{-(1/2)}}{(2\pi)^{1/2}},
	\eeq
	and since the set of the edges of $T_n$ can be split into $n$ families of parallel edges, we finally have
	\beq
		\E[\Psi_{[0,m]^n}[X]] = d^{(1)}_{0} \times d^{(2)}_{0} = -\frac{|\lambda_2|^{(1/2)}n\times m}{(2\pi)^{1/2}}.
	\eeq
	This is precisely \eqref{equn:mean} for this case.

\bibliographystyle{elsarticle-harv}
\bibliography{refs}

\begin{thebibliography}{16}
\expandafter\ifx\csname natexlab\endcsname\relax\def\natexlab#1{#1}\fi
\expandafter\ifx\csname url\endcsname\relax
  \def\url#1{\texttt{#1}}\fi
\expandafter\ifx\csname urlprefix\endcsname\relax\def\urlprefix{URL }\fi

\bibitem[{{Adler} and {Taylor}(2007)}]{Adler2007}
{Adler}, R.~J., {Taylor}, J.~E., 2007. {Random Fields and Geometry.} New York,
  NY: Springer.

\bibitem[{{Arcones}(1994)}]{arcones1994}
{Arcones}, M.~A., 1994. {Limit theorems for nonlinear functionals of a
  stationary Gaussian sequence of vectors.} {Ann. Probab.} 22~(4), 2242--2274.

\bibitem[{{Aza\"\i s} and {Wschebor}(2009)}]{Az2009}
{Aza\"\i s}, J.-M., {Wschebor}, M., 2009. Level Sets and Extrema of Random
  Processes and Fields. Hoboken, NJ: John Wiley \& Sons.

\bibitem[{{Baryshnikov} and {Ghrist}(2009)}]{Ghrist2009}
{Baryshnikov}, Y., {Ghrist}, R., 2009. {Target enumeration via Euler
  characteristic integrals.} {SIAM J. Appl. Math.} 70~(3), 825--844.

\bibitem[{Belyaev(1966)}]{Belyaev1966}
Belyaev, Y.~K., 1966. On the number of intersections of a level by a gaussian
  stochastic process. i. Theory of Probability \& Its Applications 11~(1),
  106--113.

\bibitem[{{Bobrowski} and {Borman}(2012)}]{Bobrowski2011}
{Bobrowski}, O., {Borman}, M., 2012. {Euler integration of Gaussian random
  fields and persistent homology.} {J. Topol. Anal.} 4~(1), 49--70.

\bibitem[{{Curry} et~al.(2012){Curry}, {Ghrist}, and {Robinson}}]{Ghrist2012}
{Curry}, J., {Ghrist}, R., {Robinson}, M., 2012. {Euler calculus with
  applications to signals and sensing.} In: {Advances in applied and
  computational topology. AMS short course on computational topology, New
  Orleans, LA, USA, January 4--5, 2011}. Providence, RI: American Mathematical
  Society (AMS), pp. 75--145.

\bibitem[{Estrade and Le\'{o}n(2015)}]{Jose2014}
Estrade, A., Le\'{o}n, J., 2015. A central limit theorem for the {E}uler
  characteristic of a {G}aussian excursion set. {Ann. Probab.}{ To appear}.

\bibitem[{{Houdr\'e} and {P\'erez-Abreu}(1994)}]{Houdre}
{Houdr\'e}, C., {P\'erez-Abreu}, V. (Eds.), 1994. {Chaos expansions, multiple
  Wiener-It\^o integrals and their applications. Papers of the workshop,
  Guanajuato, Mexico, July 27-31, 1992.} Boca Raton, FL: CRC Press.

\bibitem[{{Imkeller} et~al.(1995){Imkeller}, {Perez-Abreu}, and
  {Vives}}]{Imkeller95}
{Imkeller}, P., {Perez-Abreu}, V., {Vives}, J., 1995. {Chaos expansions of
  double intersection local time of Brownian motion in $\mathbb R\sp d$ and
  renormalization.} {Stochastic Processes Appl.} 56~(1), 1--34.

\bibitem[{{Kratz} and {Leon}(1997)}]{Kratz97}
{Kratz}, M., {Leon}, J., 1997. {Hermite polynomial expansion for non-smooth
  functionals of stationary Gaussian processes: Crossings and extremes.}
  {Stochastic Processes Appl.} 66~(2), 237--252.

\bibitem[{{Kratz} and {Le\'on}(2001)}]{Kratz2001}
{Kratz}, M.~F., {Le\'on}, J.~R., 2001. {Central limit theorems for level
  functionals of stationary Gaussian processes and fields.} {J. Theor. Probab.}
  14~(3), 639--672.

\bibitem[{{Major}(2014)}]{Major2014}
{Major}, P., 2014. {Multiple Wiener-It\^o integrals. With applications to limit
  theorems. 2nd ed.}, 2nd Edition. Berlin: Springer.

\bibitem[{{Nourdin} and {Peccati}(2012)}]{Nourdin2012}
{Nourdin}, I., {Peccati}, G., 2012. {Normal Approximations with Malliavin
  Calculus. From Stein's Method to Universality.} Cambridge: Cambridge
  University Press.

\bibitem[{{Nualart}(2006)}]{Nualart2006}
{Nualart}, D., 2006. {The Malliavin Calculus and Related Topics}, 2nd Edition.
  Berlin: Springer.

\bibitem[{{Yaglom}(1962)}]{Yaglom62}
{Yaglom}, A., 1962. {An Introduction to the Theory of Stationary Random
  Functions. Revised English ed. Translated and edited by Richard A.
  Silverman.} {Englewood Cliffs, NJ: Prentice-Hall, Inc. XIII, 235 p. (1962).}

\end{thebibliography}
\end{document}